\documentclass[reqno,11pt]{amsart}
\usepackage[utf8]{inputenc}    
\usepackage{setspace,tikz,xcolor,mathrsfs,listings,multicol,amssymb,amsfonts,bbm}
\usepackage{pgfplots}
\usepackage{rotating}
\usepackage[vcentermath]{youngtab}
\usepackage{tkz-base}
\usepackage{tkz-euclide}
\usepackage{tikz-cd}
\usepackage{tikz}
\usepackage{comment}
\usepackage{hyperref}
\usepackage[margin=1in,includefoot,footskip=30pt]{geometry}
\usepackage{enumerate}
\usepackage{enumitem}
\usepackage{booktabs}
\usepackage{gensymb}
\usepackage[numbers,sort&compress]{natbib}
\usepackage[centertableaux]{ytableau}
\usepackage[all,cmtip]{xy}
\usetikzlibrary{arrows,matrix,math}

\tikzset{tab/.style={matrix of math nodes,column sep=-.35, row sep=-.35,text height=7pt,text width=7pt,align=center,inner sep=2,font=\footnotesize}}

\newcommand{\p}{\hat{p}}

\newcommand{\Khat}{\widehat{\mathfrak{K}}}

\newcommand{\Span}{\operatorname{span}} 
\newcommand{\Sp}{Sp}

\newcommand{\GL}{GL}
\newcommand{\II}{\mathbb{I}}
\newcommand{\NN}{\mathbb{N}}

\newcommand{\ZZ}{\mathbb{Z}}

\newcommand{\RR}{\mathbb{R}}
\newcommand{\CC}{\mathbb{C}}
\newcommand{\sps}{\mathrm{sp}}

\newcommand{\mcK}{\mathcal{K}}

\definecolor{darkred}{rgb}{0.7,0,0} 

\definecolor{UQgold}{RGB}{196, 158, 54} 
\definecolor{UQpurple}{RGB}{73, 7, 94} 
\theoremstyle{plain}
\newtheorem{thm}{Theorem}[section]
\newtheorem{lemma}[thm]{Lemma}

\newtheorem{prop}[thm]{Proposition}

\theoremstyle{definition}
\newtheorem{dfn}[thm]{Definition}

\newtheorem{remark}[thm]{Remark}

\numberwithin{equation}{section}

\usepackage[colorinlistoftodos]{todonotes}

\begin{document}
\title{Christoffel transform and symplectic skew Howe duality}

  \author[A.~Nazarov]{Anton Nazarov}
  \address[A.~Nazarov] {Department of High Energy and Elementary
    Particle Physics, St.\ Petersburg State University, University
    Embankment, 7/9, St.\ Petersburg, Russia, 199034 and 
    Beijing Institute of Mathematical Sciences and Applications (BIMSA),
\textit{Beijing} 101408, People's Republic of China}
  \email{antonnaz@gmail.com}
 \urladdr{http://hep.spbu.ru/index.php/en/1-nazarov}

 \author[P.~Nikitin]{Pavel Nikitin}
 \address[P.~Nikitin]{
   Beijing Institute of Mathematical Sciences and Applications (BIMSA),
   Beijing 101408, People's Republic of China}
 \email{pnikitin0103@yahoo.co.uk}
 
   \author[A.~Selemenchuk]{Anton Selemenchuk}
  \address[A.~Selemenchuk] {Faculty of Mathematics and Computer Science,
    St.\ Petersburg State University, University Embankment, 7/9, St.\
    Petersburg, Russia, 199034}
  \email{selemenchyk@icloud.com}
\begin{abstract}  
For a symmetric weight $w(x)$ on a finite discrete lattice and its Christoffel transforms $x^{2}w(x),x^{2}(x^{2}w(x)),\dots$, we prove that, at each Christoffel step, the conjugated projection associated with the transformed Christoffel--Darboux kernel differs from the original orthogonal projection by a rank-one operator on the subspace of functions vanishing at the origin. This provides a general mechanism for transferring local asymptotic results from an orthogonal polynomial ensemble to its Christoffel-transformed counterpart as the lattice size tends to infinity.

As a main application, we study local fluctuations of random Young diagrams arising from skew $(\mathrm{Sp}_{2n},\mathrm{Sp}_{2k})$ Howe duality. The corresponding particle ensemble is obtained from the Krawtchouk orthogonal polynomial ensemble on a quadratic lattice by a Christoffel transform. We identify four asymptotic regimes of local fluctuations in the limit $n,k\to\infty$ with $n/k\to c\in(0,\infty)$. Besides the universal bulk fluctuations governed by the discrete sine kernel and universal Airy fluctuations at the right edge of the limit shape, we obtain the discrete Hermite kernel in the critical regime
$(k-n)/\sqrt{n+k}\longrightarrow r\in\RR$, and the discrete hard-wall sine kernel at the left corner.
\end{abstract}

\maketitle

\tableofcontents

\section{Introduction}

Discrete orthogonal polynomial ensembles form one of the main classes of determinantal point processes, appearing in random matrix theory, random tilings, non-intersecting paths, and asymptotic representation theory.

They can be defined as random $n$-point configurations on discrete lattices, with the probability of the unordered configuration $\{x_1,\ldots,x_n\}$ given by

$$
\mu^{W}_n(x_1,...,x_n)=\frac{1}{Z}\prod\limits_{i<j}(x_i-x_j)^2\prod\limits_{i=1}^n W(x_i),
$$
where $W(x)$ is a non-negative weight and $Z$ is a normalization constant. 
Their asymptotic properties are well studied, see \cite{baik2007discrete, johansson2002non, Borodin_2017} and the references therein. 

Motivated by representation-theoretic applications, the present paper studies the asymptotic behavior of Christoffel transforms of such ensembles on quadratic lattices with modified weights $x^{2d}W(x)$, $d\in\NN$. Informally, one of our main results can be stated as follows: \emph{under mild assumptions on the weight, the local asymptotics of the Christoffel-transformed discrete orthogonal polynomial ensemble coincide with those of the original ensemble}. In particular, the limiting density is the same, and known bulk and edge limits, such as the sine and Airy kernels, transfer to the Christoffel-transformed ensemble; see Proposition~\ref{As_prop_neq0}. The representation in terms of a symmetric weight also leads to a less familiar local regime governed by the discrete hard-wall sine kernel; see Proposition ~\ref{As_prop_0}.

Christoffel transformations arise naturally in several representation-theoretic models. As a main application we study local fluctuations of random Young diagrams with respect to probability measures arising from skew $(\mathrm{Sp}_{2n},\mathrm{Sp}_{2k})$ Howe duality.

Classical $(\GL_{n},\GL_{k})$ Howe duality \cite{R_Howe, howe1995perspectives} leads to Cauchy identity for Schur functions $s_{\lambda}(x):=s_{\lambda}(x_{1},x_{2},\dots)$,
$$\sum_{\lambda}s_{\lambda}(x)s_{\lambda}(y)=\prod_{i,j}\frac{1}{1-x_{i}y_{j}}.
$$
The celebrated Schur probability measure on Young diagrams 
$\lambda=(\lambda_{1},\dots,\lambda_{\ell(\lambda)})$ is introduced as~\cite{Okounkov01}
$$
\mu(\lambda)=\prod_{i,j}(1-x_{i}y_{j})s_{\lambda}(x)s_{\lambda}(y).
$$
 Schur measure admits many different generalizations, in particular, dual Cauchy identity
$$\sum_{\lambda:\ell(\lambda)\leq n,\ell(\lambda')\leq
  k}s_{\lambda}(x_{1},\dots,x_{n})
s_{\lambda'}(y_{1},\dots,y_{k})=\prod_{i=1}^{n}\prod_{j=1}^{k}(1+x_{i}y_{j}),$$
with $\lambda'$ denoting transposed Young diagram, corresponds to skew $(\GL_{n},\GL_{k})$ Howe duality on the space
$\bigwedge(\CC^{n}\otimes\CC^{k})$. Dual Schur measure was studied in great detail for different choices of specialization parameters, in particular, the convergence of random diagrams to the limit shapes was established and both local and global fluctuations were described
\cite{GTW01,GTW02II,GTW02,betea2024}. It can be naturally extended to other classical series of Lie groups. For example, the space
$\bigwedge(\CC^{2n}\otimes \CC^{k})$ has commuting actions of $\Sp_{2n}$ and $\Sp_{2k}$, so it has multiplicity-free decomposition into irreducible $\Sp_{2n}\times\Sp_{2k}$-modules. Taking characters of both sides of such a decomposition we get dual Cauchy identity
$$\sum_{\lambda\subset
  k^{n}}\sps_{\lambda}(x)\sps_{\overline{\lambda}'}(y)=\prod_{i=1}^{n}\prod_{j=1}^{k}(x_{i}+x_{i}^{-1}+y_{j}+y_{j}^{-1}),$$
where $\sps_{\lambda}(x_{1},\dots,x_{n})$ are symplectic Schur
polynomials, Young diagrams are confined to $n\times k$ rectangle and
$\overline{\lambda}$ is complement of diagram $\lambda$ inside of the
rectangle. In the present paper we consider the measure
\begin{equation}
  \label{eq:symplectic-schur-measure}
  \mu_{n,k}(\lambda|x,y)=\prod_{i=1}^{n}\prod_{j=1}^{k}\frac{1}{x_{i}+x_{i}^{-1}+y_{j}+y_{j}^{-1}}
  \cdot\sps_{\lambda}(x)\sps_{\overline{\lambda}'}(y).
\end{equation}
The appearance of symplectic Schur polynomials leads to new behavior compared with the usual Schur measures, even when only one symplectic polynomial is present in the definition of the
measure~\cite{Betea18,cuenca2024symplecticschurprocess}. But the case
\eqref{eq:symplectic-schur-measure} is much more involved, as it does not admit a simple construction in terms of free fermions. Therefore we consider only the simplest specialization $x_{i}=1, y_{j}=1$ for
all $i,j$. Symplectic Schur polynomials then give the dimensions of $\Sp_{2n}$ and $\Sp_{2k}$ irreducible representations.

Such a specialization for the classical dual Cauchy identity leads to orthogonal polynomial ensemble with Krawtchouk
polynomials~\cite{johansson2002non,borodin2007asymptotics} due to Weyl dimension formula~\cite{nazarov2024skew}. Orthogonal polynomial
ensembles are determinantal -- all correlation functions are obtained as determinants of Christoffel--Darboux correlation kernel~\cite{Simon2008}.
For symplectic Schur polynomials the dimension formula gives orthogonal polynomial ensemble with semiclassical orthogonal polynomials, due to existence of short and long roots in the root
system~\cite{NNP20,NazarovSelemenchuk2025}.  As was demonstrated in our previous paper~\cite{NazarovSelemenchuk2025} these polynomials can be obtained from Krawtchouk polynomials by Christoffel transformation.

 To the best of our knowledge, only a small number of papers address the asymptotic properties of the Christoffel-transformed polynomial ensembles. Two works directly relevant to the present setting are~\cite{lazag2025explicit} and~\cite{NazarovSelemenchuk2025}. The former  studies Christoffel transformation of the poissonized Plancherel measure and $z$-measures. In the latter, the bulk asymptotic behavior of the Christoffel-transformed Krawtchouk ensemble was analyzed using an explicit integral representation of the corresponding orthogonal polynomials. 

 In the present paper, we consider Christoffel transformations for general weights on quadratic lattices and reformulate our problem in terms of orthogonal polynomials with symmetric weights. We prove that, on the subspace of functions vanishing at the origin, the difference between the conjugated Christoffel-transformed projection and the original orthogonal projection is a rank-one operator, see Proposition~\ref{Main_Lemma}. 
This allows us to transfer the local asymptotic behaviour from the original ensemble to its Christoffel transforms at the macroscopic points away from the origin. The origin requires a separate analysis and leads to the discrete hard-wall sine kernel.
 
 \emph{As the main application, we identify four asymptotic regimes of local fluctuations for the measure \eqref{eq:symplectic-schur-measure} and determine the corresponding limits of the Christoffel--Darboux correlation kernel. In addition to the universal bulk fluctuations governed by the discrete sine kernel and Airy fluctuations~\cite{TracyWidom1994,deiftorthogonal} at the right soft edge of the limit shape, we obtain the discrete Hermite kernel, introduced in \cite{borodin2007asymptotics}, in the critical regime
$(k-n)/\sqrt{n+k}\longrightarrow r\in\RR,$ and the discrete hard-wall sine kernel on the left corner.} The discrete hard-wall sine kernel has previously appeared in asymptotic representation theory in the work of A.~Borodin and J.~Kuan \cite{borodin2010random}.
 It was also studied in connection with XX chains~\cite{Fagotti2011} and its continuous analogue is well known in random matrix theory \cite{Erhardt2007Dyson, verbaarschot1994,verbaarschot2000}, see also  \cite[Sections 3.1 and
7.2.7]{forrester2010log}.

The paper is organized as follows. In Section~\ref{sec:determ-point-proc} we recall that the probability measure~\eqref{eq:young-measure} is determinantal, with a Christoffel--Darboux correlation kernel on the quadratic lattice. Section~\ref{sec:outline-method} develops the rank-one comparison and the resulting transfer principles for iterated Christoffel transformations of symmetric weights. In Section~\ref{sec:examples} we apply these results first to the modified Krawtchouk ensemble arising from skew $(\Sp_{2n},\Sp_{2k})$ Howe duality and then to symmetric Hahn ensembles, including a discrete Laguerre degeneration. In Conclusion we outline directions for further research.
Appendix~\ref{app:Hahn_edge_parameters} derives the Hahn density and edge parameters, while Appendix~\ref{sec:spectral_first_christoffel} explains the spectral-projection interpretation of the first Christoffel transform.

\section{Determinantal Point Process from Skew Howe Duality for Symplectic Groups}
\label{sec:determ-point-proc}

By the skew Howe duality for the pair \((\mathrm{Sp}_{2n},\mathrm{Sp}_{2k})\)  acting on the space 
$\bigwedge\Bigl(\mathbb{C}^{2n}\otimes\mathbb{C}^{k}\Bigr)
\cong
\bigwedge\Bigl(\mathbb{C}^{n}\otimes\mathbb{C}^{2k}\Bigr)$
we have the decomposition
\begin{align*}
\bigwedge\Bigl(\mathbb{C}^{2n}\otimes\mathbb{C}^{k}\Bigr)
\cong \bigoplus_{\lambda\subset n\times k} V_{\mathrm{Sp}_{2n}}(\lambda)\otimes V_{\mathrm{Sp}_{2k}}(\overline{\lambda}'),
\end{align*}
where \(\lambda\) is a Young diagram fitting into an \(n\times k\) rectangle, $\overline\lambda'$ is its transposed
complement in the rectangle, and by $V_{\text{Sp}_{2n}}(\lambda)$ we denote
the irreducible representation corresponding to the diagram~$\lambda$~\cite{howe1995perspectives,cheng2012dualities}. This duality induces a probability measure on Young diagrams given by
\begin{align}
\mu_{n,k}(\lambda) = \frac{\dim V_{\mathrm{Sp}_{2n}}(\lambda)\cdot\dim V_{\mathrm{Sp}_{2k}}(\overline{\lambda}')}{2^{2nk}},
\label{eq:young-measure}
\end{align}
with the normalization constant \(2^{2nk}=\dim \bigwedge\Bigl(\mathbb{C}^{2n}\otimes\mathbb{C}^{k}\Bigr)\). After a rotation (the “Russian” convention) and an appropriate scaling, a diagram \(\lambda\) can be encoded by a particle configuration on a lattice
$\{1,...,n+k\}$ 
via the coordinates
$$
a_i = \lambda_i + n - i + 1,\quad i=1,\dots,n.
$$
This bijective map allows us to transfer the measure to the space of point configurations on $\{1,...,n+k\}$.
The measure can be written in the following form using Weyl dimension formula and Lindstr\"om-Gessel-Viennot lemma (see \cite{NNP20}):  
\begin{align}
	\mu_{n,k}(a_1,...,a_n)=\frac{1}{Z_{n,k}}\prod\limits_{i<j}(a_i^2-a_j^2)^2\prod\limits_{i=1}^n a_i^2\binom{2N}{N+a_i},
\end{align}
where the $Z_{n,k}$ is a normalization constant, $N=n+k$. 

This motivates the following definition. 
For $N\in\NN$ and $d\in\ZZ_{\geq0}$, take a positive weight $W(x)>0$, defined on a set $\{i^2\mid i\in\ZZ_{\geq0},\ i\leq N\}$, and introduce a new weight
\begin{equation*}
    \widetilde{W}^{(d)}(x^2)=x^{2d}W(x^2)
\end{equation*}
and a measure 
\begin{align}
	\mu^{W,d}_n(a_1,...,a_n)=\frac{1}{Z_{n}^{(d)}}\prod\limits_{i<j}(a_i^2-a_j^2)^2\prod\limits_{i=1}^n \widetilde W^{(d)}(a_i^2),
\end{align}
where $Z_{n}^{(d)}$ is a normalization constant. The main motivation for the present paper is the asymptotic analysis of the random point configurations $(a_1,...,a_n)$ with respect to such measures.

We introduce a kernel $\mathfrak{K}^{(d)}_n(x,y)$ as the Christoffel--Darboux kernel of the form
\begin{align}\label{measure_mod}
	\mathfrak{K}^{(d)}_n(x,y)=\sqrt{\widetilde{W}^{(d)}(x^2)\widetilde{W}^{(d)}(y^2)} \sum\limits_{\ell=0}^{n-1}\pi^{(d)}_\ell(x^2)\pi^{(d)}_\ell(y^2),
\end{align}
where the polynomials $\{\pi^{(d)}_\ell\}$ on the quadratic lattice satisfy the orthogonality relations
\begin{align*}
\sum\limits_{x=0}^{N}\pi^{(d)}_m(x^2)\pi^{(d)}_\ell(x^2)\widetilde{W}^{(d)}(x^2)=\delta_{m\ell}.
\end{align*}

Then the measure $\mu^{W,d}_n$ admits a determinantal representation
\begin{align}
  \mu^{W,d}_n(a_1,\dots,a_n) = \det\big[\mathfrak{K}^{(d)}_{n}(a_i,a_j)\big]_{i,j=1}^{n}.
\end{align}
Such a structure means that the measures in question belong to the class of determinantal point processes, see \cite{Mehta, Soshnikov, baik2007discrete}, and the limiting behavior of the random point configurations is governed by the asymptotics of the kernel $\mathfrak{K}^{(d)}_{n}(x,y)$. To be more precise,
on a countable discrete space, 
pointwise convergence of the kernels implies the weak convergence of the corresponding processes, see Proposition~4.1 in~\cite{Borodin_2017}.

In the next section we study the convergence properties of the Christoffel--Darboux kernels $\mathfrak{K}^{(d)}_{n}(x,y)$.

\section{Main results}
\label{sec:outline-method}
This section studies the convergence of the Christoffel--Darboux kernels associated with iterated Christoffel transformations on a quadratic lattice. It is convenient to pass to the symmetric linear lattice $\mathfrak X=\{-N,\ldots,N\}$. Orthogonal polynomials on the quadratic lattice can then be represented by even-degree orthogonal polynomials on $\mathfrak X$ with respect to the symmetrized weights $w^{(d)}(x)$,
\begin{align}
	w^{(d)}(x) = \frac{1}{2}\widetilde{W}^{(d)}(x^2)  = \frac{1}{2}x^{2d}W(x^2),\quad x\neq 0 \text{ or } d\neq 0; \qquad w^{(0)}(0) = W(0).
\end{align}

Introduce monic orthogonal polynomials $\{P_m(x)\}_{m=0}^{2N}$ with respect to the symmetric weight $w(x)=w^{(0)}(x)$,
\begin{align*}
	\sum\limits_{x\in \mathfrak{X}}P_m(x)P_\ell(x)w(x)=
    \delta_{m\ell}h_m^2,\quad h_m=||P_m||, 
\end{align*}
and the basis of orthonormal functions $\{p_m(x)\}_{m=0}^{2N}$ in $\ell^2(\mathfrak{X})$,
\begin{align}\label{lat_X}
	p_m(x)=\frac{P_m(x)}{||P_m||}\sqrt{w(x)},\quad x\in \mathfrak{X}=\left\{-N, -N+1,\dots,N-1,N\right\}.
\end{align}
For a symmetric weight $w(x)$, even degree polynomials are even functions, therefore
\begin{align*}
	\sum\limits_{x\in \mathfrak{X}}P_{2m}(x)P_{2\ell}(x)w(x)=
    \sum\limits_{y=0}^N P_{2m}(y)P_{2\ell}(y)W(y^2)=
    \delta_{m\ell}h_{2m}^2.
\end{align*}
Hence we have 
$p_{2m}(x) = \pi^{(0)}_m(x^2)\sqrt{w(x)}$, $0\le m\le N$, 
and we can rewrite the kernel $\mathfrak{K}^{(0)}_{n}(x,y)$ in the form
\begin{align}\label{CD-kernel_0}
	\mathfrak{K}^{(0)}_n(x,y) = 
    &2\eta(x)\eta(y)\sum\limits_{\ell=0}^{n-1}p_{2\ell}(x)p_{2\ell}(y), \quad \eta(x)=\frac{1}{\sqrt{1+\delta_{x,0}}}, \quad x,y\in \mathbb{Z}.
\end{align}
It would be convenient to assume that $\eta(x)$ is defined for $x\in \mathbb{R}$,
\begin{align*}
	\eta(x)=\begin{cases}
1,\quad x\neq 0,\\
1/\sqrt{2}, \quad x=0.
\end{cases}
\end{align*} 
\begin{remark}
The normalization factor
$
\eta(x)
$
reflects the fact that the map $x\mapsto x^2$ identifies the points
$x$ and $-x$. More precisely, for any even functions $f,g$ on the
symmetric lattice $\mathfrak X=\{-N,\dots,N\}$ one has
\[
\sum_{x\in\mathfrak X}f(x)g(x)w(x)
=
2\sum_{y=1}^{N}f(y)g(y)w(y)+f(0)g(0)w(0),
\]
since every positive point of the quadratic lattice has exactly two
preimages, while the origin has only one. Consequently, the
orthonormal functions on the quadratic lattice are obtained from the
restrictions of the even orthonormal functions on $\mathfrak X$ by
renormalizing the value at the origin with the factor $\eta(x)$, which
explains the representation~\eqref{CD-kernel_0}. After the first
Christoffel transformation, however, the modified weight
\[
w^{(d)}(x)=x^{2d}w(x), \qquad d>0,
\]
vanishes at the origin. Therefore the point $0$ no longer contributes
to the scalar product, and no additional normalization is required.
This is why the factor $\eta(x)$ disappears from the kernel
representation (see \eqref{CD-kernel_1} below). Notice that this remains true for
all subsequent Christoffel transformations, since the weight continues
to vanish at the origin.
\end{remark}
Let $\{P^{(d)}_m(x)\}_{m=0}^{2N-1}$ be the monic orthogonal polynomials with respect to the weight function $x^{2d}w(x),\;d>0$, and denote the corresponding orthonormal functions by 
$$
p^{(d)}_m(x)=\frac{P^{(d)}_m(x)}{||P^{(d)}_m||_{x^{2d}w(x)}}x^d\sqrt{w(x)}.
$$ 
Similarly to~\eqref{CD-kernel_0} we obtain, for $d>0$,
\begin{align}\label{CD-kernel_1}
	\mathfrak{K}^{(d)}_n(x,y) = 
    2\sum\limits_{\ell=0}^{n-1}p^{(d)}_{2\ell}(x) p^{(d)}_{2\ell}(y).
\end{align}
To simplify the notation we set $\hat{P}_m(x)=P^{(1)}_m(x)$ and $\hat{p}_m(x)=p^{(1)}_m(x)$.
It follows from the orthogonality conditions and the symmetry of the weight that the polynomials $\{\hat{P}_m(x)\}$ can be expressed via the polynomials $\{P_m\}$ by the following relations (see \cite[Theorem 2.7.1]{Ismail_2005}):
\begin{align} \label{hatPm_p}
	&x^2\hat{P}_m(x)=P_{m+2}(x)+S_mP_m(x),\\ \notag
	& S_{2l}=-P_{2l+2}(0)/P_{2l}(0)=\beta_{2l+1}, \quad S_{2l+1}=-P'_{2l+3}(0)/P'_{2l+1}(0)=\hat{\beta}_{2l+1},
\end{align}
where the polynomials $\{\hat{P}_m\}$ and $\{P_{m}\}$ satisfy three term recurrence relations of the form
\begin{align}\label{TTR}
  &xP_m(x)=P_{m+1}(x)+\beta_mP_{m-1}(x),
  &x\hat{P}_m(x)=\hat{P}_{m+1}(x)+\hat{\beta}_m\hat{P}_{m-1}(x).
\end{align}
For general orthogonal polynomials we have $\beta_m = h_m^2 h_{m-1}^{-2}$, and it leads to the symmetric version of the three term recurrence relations:
\begin{align}\label{TTR_sym}
  &xp_m(x) = \sqrt{\beta_{m+1}}p_{m+1}(x)+\sqrt{\beta_m} p_{m-1}(x).
\end{align}

It follows directly from the equations \eqref{hatPm_p} and \eqref{TTR} that the modified polynomials for even indices become 
\begin{align}\label{even_hat}
	\hat{P}_{2m}(x)=\frac{P_{2m+1}(x)}{x},\quad \hat{h}_{2m}=h_{2m+1},\quad \hat{\beta}_{2m}=\frac{h_{2m}^2}{\hat{h}_{2m-1}^2}\beta_{2m+1}.
\end{align} 
Let us denote the even and odd subspaces of $\ell^2(\mathfrak{X})$ by $\ell^2_+(\mathfrak{X})$ and $\ell^2_-(\mathfrak{X})$:
\begin{align}
	\ell^2_+(\mathfrak{X}) = \langle p_{2i}(x)\rangle_{i=0}^{N},\quad \ell^2_-(\mathfrak{X})=\langle p_{2i+1}(x)\rangle_{i=0}^{N-1}.
\end{align}
Let 
$$\mathbb{K}_{2n}(x,y)=\sum\limits_{i=0}^{2n-1}p_i(x)p_i(y)=\sqrt{\beta_{2n}}
		\frac{
			p_{2n}(x)p_{2n-1}(y)
			-
			p_{2n-1}(x)p_{2n}(y)
		}{
			x-y
		}$$ 
be the Christoffel--Darboux kernel for $w(x)$. 
The kernel of the restriction of the corresponding projection operator to the subspace $\ell^2_+(\mathfrak{X})$ of even functions is given by $\sum\limits_{\ell=0}^{n-1}p_{2\ell}(x)p_{2\ell}(y)$, and it coincides with the kernel $\mathfrak{K}^{(0)}_n(x,y)$ up to a factor $2\eta(x)\eta(y)$:
\begin{equation} \label{sym_ker}
\mathfrak{K}^{(0)}_n(x,y)
= 2\eta(x)\eta(y)\sum\limits_{\ell=0}^{n-1}p_{2\ell}(x)p_{2\ell}(y)
=
\eta(x)\eta(y)\left[\mathbb{K}_{2n}(x,y)+\mathbb{K}_{2n}(-x,y)\right].
\end{equation}
Further, we introduce the following subspaces of polynomials with zero constant term:
\begin{align*}
&\ell^2_0(\mathfrak{X}) = 
    \langle \hat{p}_m(x)\rangle_{m=0}^{2N-1} = 
    \{p\in \ell^2(\mathfrak{X}) \mid p(0)=0\}, \\
&\ell^2_{+,0}(\mathfrak{X}) = 
   \ell^2_0(\mathfrak{X}) \cap \ell^2_+(\mathfrak{X}), \quad
\ell^2_{-,0}(\mathfrak{X}) = 
    \ell^2_0(\mathfrak{X}) \cap \ell^2_-(\mathfrak{X}),
\end{align*}
and the following operators: an operator $\mcK_{2n}$ of the orthogonal projection onto the subspace $\langle p_{2m}(x)\rangle_{m=0}^{n-1}$ in $\ell^2(\mathfrak{X})$,
$$
\mcK_{2n}[f](x) =\sum\limits_{y\in \mathfrak{X}}\sum\limits_{\ell=0}^{n-1}p_{2\ell}(x)p_{2\ell}(y)f(y)= \sum\limits_{y\in \mathfrak{X}}\frac{1}{2\eta(x)\eta(y)}\mathfrak{K}^{(0)}_n(x,y)f(y),  
$$
an operator $\hat\mcK_{2n}$ of the orthogonal projection onto the subspace $\langle \hat p_{2k}(x)\rangle_{k=0}^{n-1}$ in $\ell^2_0(\mathfrak{X})$,

\begin{align}\label{sym_ker_con_1}
\hat\mcK_{2n}[f](x) = \frac12\sum\limits_{y\in \mathfrak{X}}\mathfrak{K}^{(1)}_n(x,y)f(y),  
\end{align}
and a multiplication operator $X$,
$$
X[f](x) = xf(x).
$$
Note that $X$ is invertible on $\ell^2_0(\mathfrak{X})$.

The following proposition shows that the difference between the non-orthogonal projection $X\hat{\mcK}_{2n}X^{-1}$ and the orthogonal projection $\mcK_{2n+2}$ has rank one on the subspace $\ell^2_0(\mathfrak{X})$ of functions vanishing at the origin.

\begin{prop}
  \label{Main_Lemma}
On the subspace $\ell^2_0(\mathfrak X)$, the operator 
$$
X\hat{\mcK}_{2n}X^{-1}-\mcK_{2n+2}
$$
is represented, for $y\neq0$, by the kernel 
\begin{equation}\label{eq:Main_Lemma}
\frac{x}{2y}\mathfrak K_n^{(1)}(x,y) - 
\frac{1}{2\eta(x)}\mathfrak K_{n+1}^{(0)}(x,y) =
-\frac{\sqrt{\beta_{2n+1}}}{y} p_{2n}(x)p_{2n+1}(y).
\end{equation}
In particular, this operator has rank one.
Moreover, if $w(0)=0$, it simplifies to 
\begin{equation}\label{eq:Main_Lemma2}
\frac{x}{2y}\mathfrak K_n^{(1)}(x,y) - 
\frac{1}{2}\mathfrak K_{n+1}^{(0)}(x,y) =
-\frac{\sqrt{\beta_{2n+1}}}{y} p_{2n}(x)p_{2n+1}(y).
\end{equation}
\end{prop}

\begin{proof}
 It follows from \eqref{TTR_sym} and \eqref{even_hat} that
 \begin{align}\label{basis_tr}
 	&x\p_{2k}(x)=p_{2k+2}(x)\sqrt{\beta_{2k+2}}+p_{2k}(x)\sqrt{\beta_{2k+1}}, \\
 	&x\p_{2k+1}(x)=p_{2k+3}(x)\sqrt{\hat{\beta}_{2k+2}}+p_{2k+1}(x)\sqrt{\hat{\beta}_{2k+1}}.
\end{align}
We now introduce the basis
\begin{align}
\left\{ \frac{x\hat{p}_{2m}(x)}{\sqrt{\beta_{2m+2}}} \right\}_{m=0}^{N-1}
\end{align} 
in $\ell^2_{+,0}(\mathfrak{X})$ and a basis $\{x\hat{p}_{2m+1}(x)\}_{m=0}^{N-1}$ in
$\ell^2_{-,0}(\mathfrak{X})$. Define the operator $\Delta_{2n} = X\hat{\mathcal{K}}_{2n}X^{-1}-\mcK_{2n+2}$, 
restricted to the subspace 
$\ell^2_0(\mathfrak{X})$.
We have
 \begin{align}\label{Delta}
 	\Delta_{2n}[x\hat{p}_m(x)]=(X\hat{\mathcal{K}}_{2n}X^{-1}-\mcK_{2n+2})[x\hat{p}_m(x)]=\begin{cases}&0,\quad m\neq 2n,\\
 	&-p_{2n}(x)\sqrt{\beta_{2n+1}},\quad m=2n.
 	\end{cases}
 \end{align}
 Consequently, $\Delta_{2n}$ vanishes identically on $\ell^2_{-,0}(\mathfrak{X})$, and its range is the one-dimensional subspace spanned by $p_{2n}$. Its kernel is therefore
\begin{align}
    \Delta_{2n}(x,y)=-\sqrt{\beta_{2n+1}}\;p_{2n}(x)\frac{\p_{2n}(y)}{y}=-\sqrt{\beta_{2n+1}}\;p_{2n}(x)\frac{p_{2n+1}(y)}{y},
\end{align}
where the last equality follows from~\eqref{even_hat}, which completes the proof.
\end{proof}

\begin{remark}
We will be interested in the local asymptotics of the kernels
$\mathfrak K_n^{(d)}(x_n,y_n)$ as $n\to\infty$, where $x_n=un + o(n)$ and $y_n=u n + o(n)$, for some fixed $d\ge0$ and $u\neq0$. Since ${x_n}/{y_n}\longrightarrow 1$, the kernel $\mathfrak K_n^{(d)}(x_n,y_n)$ and its conjugated version $\frac{x_n}{y_n}\mathfrak K_n^{(d)}(x_n,y_n)$ have the same limit whenever either of the two limits exists. 

Therefore, by iterating Proposition~\ref{Main_Lemma}, if all the rank-one terms arising in the successive applications of \eqref{eq:Main_Lemma} tend to $0$, 
and if $\mathfrak K_{n}^{(0)}(x_n,y_n)$ converges, then, for every fixed $d>0$, the kernel $\mathfrak K_{n-d}^{(d)}(x_n,y_n)$ converges to the same limit.
Proposition~\ref{As_prop_neq0} and Remark~\ref{rem:fixed_shift} below explain why the finite shift of indices appearing here does not affect the limiting kernel in the applications.
\end{remark}
The behavior at the origin is different. We start with the corresponding limiting formulas for $d=0$ and $d=1$.
Let $[t]$ denote a symmetric nearest-integer rounding, that is, an integer-valued map satisfying
$[-t] = -[t].$
\begin{prop}\label{As_prop_0}
Let $\lim\limits_{N\to \infty}n/N=\gamma\in(0,1)$.
\begin{enumerate}
\item 
If there exists a limiting kernel $\mathbb{K}_{\lim}$
such that for all $A,B\in\ZZ$ we have
$$
\lim_{N\to \infty} \mathbb{K}_{2n}(A,B)= \mathbb{K}_{\lim}(A,B),
$$
then for any $A,B\in\ZZ$ 
	\begin{align*}
		&\lim_{N\to \infty} \mathfrak{K}^{(0)}_n(A,B) =  
        \eta(A)\eta(B)
\left[
\mathbb K_{\lim}(A,B)
+
\mathbb K_{\lim}(-A,B)
\right],
\\
 		&\lim_{N\to \infty} \mathfrak{K}^{(1)}_n(A,B) = 
        \mathbb{K}_{\lim}(A,B)-\mathbb{K}_{\lim}(-A,B);       
	\end{align*} 
\item 
If there exist a sequence $\{D_n\}$, with $D_n = o(n)$, $D_n\to\infty$,
and a limiting kernel $\mathbb{K}_{\lim}$
such that for all $A,B\in\RR$ we have
$$
\lim_{N\to \infty}D_n \mathbb{K}_{2n}([ D_n A],[ D_n B])= \mathbb{K}_{\lim}(A,B),
$$
then for any $A,B\in\RR$ 
	\begin{align*}
		&\lim_{N\to \infty} D_n\mathfrak{K}^{(0)}_n([ D_n A ], [ D_n B ]) = \eta(A)\eta(B)\left[ \mathbb{K}_{\lim}(A,B)+\mathbb{K}_{\lim}(-A,B)\right],
\\
 		&\lim_{N\to \infty} D_n\mathfrak{K}^{(1)}_n([ D_n A ], [ D_n B ]) = 
\mathbb K_{\lim}(A,B)-\mathbb K_{\lim}(-A,B).    
	\end{align*} 
\end{enumerate}
\end{prop}
\begin{proof}
We first prove part~(1). By~\eqref{sym_ker} we have
\begin{align}\label{eq:K^0_sum}
\mathfrak K_n^{(0)}(x,y)
=
\eta(x)\eta(y)[\mathbb K_{2n}(x,y)
+
\mathbb K_{2n}(-x,y)].
\end{align}
Therefore, for any fixed integers $A,B$,
\begin{align}
\mathfrak K_n^{(0)}(A,B)
=
\eta(A)\eta(B)[\mathbb K_{2n}(A,B)
+
\mathbb K_{2n}(-A,B)].
\end{align}
Passing to the limit and using the assumed convergence of
$\mathbb K_{2n}$, we obtain
\begin{align}
\lim_{N\to\infty}
\mathfrak K_n^{(0)}(A,B)
=
\eta(A)\eta(B)[\mathbb K_{\mathrm{lim}}(A,B)
+
\mathbb K_{\mathrm{lim}}(-A,B)].
\end{align}
Next, by~\eqref{CD-kernel_1} and~\eqref{even_hat}, we have
\begin{align}\label{eq:K^1_diff}
\mathfrak K_n^{(1)}(x,y)
=
2\sum_{\ell=0}^{n-1}
\hat p_{2\ell}(x)\hat p_{2\ell}(y)
=
2\sum_{\ell=0}^{n-1}
p_{2\ell+1}(x)p_{2\ell+1}(y)=
\mathbb K_{2n}(x,y)
-
\mathbb K_{2n}(-x,y).
\end{align}
Therefore, the limit is
\begin{align}
\lim_{N\to\infty}
\mathfrak K_n^{(1)}(A,B)
=
\mathbb K_{\mathrm{lim}}(A,B)
-
\mathbb K_{\mathrm{lim}}(-A,B),
\end{align}
which proves part~(1). 

We now prove part~(2).
Again using~\eqref{eq:K^0_sum} and substituting
$x=[D_nA]$ and
$y=[ D_nB]$,
we get
\begin{align*}
\begin{aligned}
D_n
\mathfrak K_n^{(0)}
([D_nA],
[D_nB]) =
\eta([D_nA])\eta([D_nB])
D_n[\mathbb K_{2n}
([D_nA],
[D_nB]) +
\mathbb K_{2n}
(-[D_nA],
[D_nB])].
\end{aligned}
\end{align*}
By the symmetry of the rounding,
$$
-[D_nA]
=
[-D_nA]
=
[D_n(-A)].
$$
Hence the assumed convergence of the rescaled kernels applies to the two pairs $(A,B)$ and $(-A,B)$, respectively. Since $D_n\to\infty$, for every fixed $A,B\in\RR$ we have
$$
\eta([D_nA])=\eta(A),
\qquad
\eta([D_nB])=\eta(B)
$$
for all sufficiently large $n$. Thus,
$$
\lim_{N\to\infty} D_n \mathfrak K_n^{(0)}([D_nA], [D_nB]) = \eta(A)\eta(B)[
\mathbb K_{\mathrm{lim}}(A,B) +
\mathbb K_{\mathrm{lim}}(-A,B)].
$$
Similarly, using again~\eqref{eq:K^1_diff} and the symmetry of the rounding, we obtain
\begin{align*}
\lim_{N\to\infty}
D_n
\mathfrak K_n^{(1)}([D_nA],[D_nB])=
\mathbb K_{\mathrm{lim}}(A,B) -
\mathbb K_{\mathrm{lim}}(-A,B).
\end{align*}
This completes the proof.
\end{proof}
Away from the origin, the rank-one comparison leads to a particularly simple transfer principle. If the localization center stays at a nonzero macroscopic distance from zero and the local asymptotics of
the ordinary kernel are stable under bounded shifts of the projection rank, then every fixed Christoffel transform has the same local limiting kernel.

We will use this principle in two types of local scaling. For discrete local limits, we take $x_N(A)=\xi_N+A$, where
$\{\xi_N\}$ is integer-valued, $\xi_N/N\to v\neq0$, and
$A\in\ZZ$, $\ZZ_{\geq0}$, or $\ZZ_{\leq0}$. This corresponds to the choice $D_N=1$ in the next proposition. For local limits on a growing scale, we take
\[
x_N(A)=\left\lfloor \xi_N+D_NA\right\rfloor,
\qquad A\in\RR,
\]
where $\xi_N/N\to v\neq0$ and $D_N=o(N)$.
In both cases $x_N(A)/N\to v$ for every fixed $A$.

The next proposition treats both situations simultaneously, and in fact only uses this last macroscopic localization property.

\begin{prop}\label{As_prop_neq0}
Let $N\to\infty$, and let $n=n(N)$ be a sequence of positive
integers such that
\[
n\longrightarrow\infty,
\qquad
N-n\longrightarrow\infty.
\]
Let $v\in[-1,1]\setminus\{0\}$, $\mathbb X\in
\{\ZZ,\ZZ_{\geq0},\ZZ_{\leq0},\RR\}$, and let $\{D_N\}$ be a positive sequence. Suppose that there is a sequence of maps
\[
x_N:\mathbb X\longrightarrow\mathfrak X=\{-N,\ldots,N\}
\]
such that, for every fixed $A\in\mathbb X$, 
${x_N(A)}/{N}\longrightarrow v$.

Suppose that there exists a limiting kernel $\mathbb K_{\lim}$ on $\mathbb X^2$ such that, for every fixed $j\in\ZZ$ and all fixed $A,B\in\mathbb X$,
\[
\lim_{N\to\infty}
D_N\,
\mathbb K_{2n+j}
\bigl(x_N(A),x_N(B)\bigr)
=
\mathbb K_{\lim}(A,B).
\]
Then, for every fixed $d\in\ZZ_{\geq0}$ and all
$A,B\in\mathbb X$,
\[
\lim_{N\to\infty}
D_N\,
\mathfrak K_n^{(d)}
\bigl(x_N(A),x_N(B)\bigr)
=
\mathbb K_{\lim}(A,B).
\]
\end{prop}

\begin{proof}
Fix $A,B\in\mathbb X$, and set $x_N=x_N(A)$ and $y_N=x_N(B)$.
By the assumptions, $x_N/N\to v$ and $y_N/N\to v$, where
$v\neq0$. Hence $x_N/y_N\to1$ and
$(x_N+y_N)/N\to2v\neq0$. In particular, $x_N,y_N\neq0$ for all
sufficiently large $N$, so that $\eta(x_N)=\eta(y_N)=1$.

We first note that the recurrence coefficients at every Christoffel
level satisfy $\beta_m^{(d)}\leq N^2$.
Indeed, multiplication by $x$ on the support
$\mathfrak X\subset[-N,N]$ has operator norm at most $N$, while the three-term recurrence gives
$\sqrt{\beta_{m+1}^{(d)}}=
\langle xp_m^{(d)},p_{m+1}^{(d)}\rangle$. Thus
$\sqrt{\beta_{m+1}^{(d)}}\leq N$.

Next, the assumed convergence under fixed shifts of the kernel rank
implies that, for every fixed $i\in\ZZ$ and $A\in\mathbb X$,
\[
\sqrt{D_N}\,p_{2n+i}(x_N(A))\longrightarrow0.
\]
Indeed,
$D_Np_{2n+i}(x_N(A))^2$ is equal to
\[
D_N\mathbb K_{2n+i+1}(x_N(A),x_N(A))
-
D_N\mathbb K_{2n+i}(x_N(A),x_N(A)),
\]
and both terms converge to $\mathbb K_{\lim}(A,A)$.
The assumptions $n\to\infty$ and $N-n\to\infty$ ensure that all
fixed shifted indices occurring below are admissible for sufficiently
large $N$.

We shall prove the slightly stronger statement that, for every fixed
$d,\ell\in\ZZ_{\geq0}$,
\[
D_N\mathfrak K_{n+\ell}^{(d)}(x_N,y_N)
\longrightarrow
\mathbb K_{\lim}(A,B).
\]

We first consider $d=0$. By \eqref{sym_ker},
\[
D_N\mathfrak K_{n+\ell}^{(0)}(x_N,y_N)
=
D_N\mathbb K_{2n+2\ell}(x_N,y_N)
+
D_N\mathbb K_{2n+2\ell}(-x_N,y_N).
\]
The first term converges to $\mathbb K_{\lim}(A,B)$ by the
assumptions. For the reflected term, the Christoffel--Darboux formula
gives
\[
\mathbb K_{2n+2\ell}(-x_N,y_N)
=
\sqrt{\beta_{2n+2\ell}}\,
\frac{
p_{2n+2\ell}(-x_N)p_{2n+2\ell-1}(y_N)
-
p_{2n+2\ell-1}(-x_N)p_{2n+2\ell}(y_N)
}{
-x_N-y_N
}.
\]
Since $\sqrt{\beta_{2n+2\ell}}\leq N$ and
$|x_N+y_N|\asymp N$, the prefactor is $O(1)$. By parity and the
estimate above, each product in the numerator becomes $o(1)$ after
multiplication by $D_N$. Hence
$D_N\mathbb K_{2n+2\ell}(-x_N,y_N)\to0$, and the required limit
follows for $d=0$.

It remains to propagate the pointwise estimate through the successive
Christoffel transformations. We claim that, for every fixed
$d\in\ZZ_{\geq0}$, $i\in\ZZ$, and $A\in\mathbb X$,
\[
\sqrt{D_N}\,p_{2n+i}^{(d)}(x_N(A))\longrightarrow0.
\]
The case $d=0$ was proved above. Assume the assertion holds at level
$d$. The exact Christoffel relation
$p_{2m}^{(d+1)}=p_{2m+1}^{(d)}$ gives the assertion immediately for
even indices. For odd indices,
\[
x\,p_{2m+1}^{(d+1)}(x)
=
\sqrt{\beta_{2m+2}^{(d+1)}}\,p_{2m+3}^{(d)}(x)
+
\sqrt{\beta_{2m+1}^{(d+1)}}\,p_{2m+1}^{(d)}(x).
\]
Since $\sqrt{\beta_k^{(d+1)}}\leq N$ for every $k$ and $|x_N(A)|\asymp N$, we have
$\sqrt{\beta_k^{(d+1)}}/|x_N(A)|=O(1)$.
Taking $m=n+r$, with $r$ fixed, and $x=x_N(A)$, dividing the preceding identity by $x$, multiplying by $\sqrt{D_N}$, and using the induction hypothesis therefore gives
$\sqrt{D_N}\,p_{2n+2r+1}^{(d+1)}(x_N(A))\to0$.

We now proceed by induction on $d$ in the kernel asymptotics. Suppose that, for some fixed $d\geq0$,
$D_N\mathfrak K_{n+\ell}^{(d)}(x_N,y_N)\to
\mathbb K_{\lim}(A,B)$ for every fixed $\ell\geq0$.
Applying Proposition~\ref{Main_Lemma} to the weight $w^{(d)}$, with
$n$ replaced by $n+\ell$, gives
\[
\frac{x_N}{y_N}
\mathfrak K_{n+\ell}^{(d+1)}(x_N,y_N)
-
\mathfrak K_{n+\ell+1}^{(d)}(x_N,y_N)
=
-\frac{
2\sqrt{\beta_{2n+2\ell+1}^{(d)}}
}{
y_N
}
p_{2n+2\ell}^{(d)}(x_N)
p_{2n+2\ell+1}^{(d)}(y_N).
\]
For $d=0$, the same identity holds here because
$\eta(x_N)=1$ for sufficiently large $N$.
After multiplication by $D_N$, the right-hand side tends to zero:
$\sqrt{\beta_{2n+2\ell+1}^{(d)}}/|y_N|\leq N/|y_N|=O(1)$, while
\[
D_N
p_{2n+2\ell}^{(d)}(x_N)
p_{2n+2\ell+1}^{(d)}(y_N)
\longrightarrow0
\]
by the pointwise estimate proved above. Since $x_N/y_N\to1$, the
induction hypothesis with $\ell+1$ gives
$D_N\mathfrak K_{n+\ell}^{(d+1)}(x_N,y_N)\to
\mathbb K_{\lim}(A,B)$.

The induction is complete. Taking $\ell=0$ proves the proposition.
\end{proof}

\begin{remark}\label{rem:fixed_shift}
The assumption of convergence under fixed shifts of the kernel rank
fits naturally into the asymptotic framework of
Baik--Kriecherbauer--McLaughlin--Miller~\cite{baik2007discrete}.
Indeed, in \cite[\S1.3.1]{baik2007discrete} the degree of the
orthogonal polynomial is taken in the form $k=cN+\lambda$, where
$c\in(0,1)$ is fixed and $\lambda$ remains bounded as $N\to\infty$.
Thus a fixed shift $k\mapsto k+j$ only changes the bounded parameter
$\lambda$. In particular, their bulk and soft-edge kernel asymptotics
are naturally stable under the fixed rank shifts required in
Proposition~\ref{As_prop_neq0}.

The same type of stability is also natural in other approaches to
local asymptotics, including the spectral-projection framework of
Borodin--Olshanski~\cite{borodin2007asymptotics, Borodin_2017} and fermionic formulations of the correlation
kernel~\cite{Okounkov01}. We will make repeated use of these connections in the next
section when verifying the assumptions of
Proposition~\ref{As_prop_neq0} in concrete asymptotic regimes.
\end{remark}

We will also use the following dual picture in the study of the examples in the next section.
\begin{dfn}[Dual-ensemble kernel]\label{def:dual}
Let $w$ be a weight on the symmetric lattice $\mathfrak{X}=\{-N,...,N\}$, let $\mathfrak{K}^{(d)}_{n}$ be the associated ($d$-fold Christoffel-transformed) correlation kernel on the quadratic lattice, written in the square-root
coordinates $x,y\in\{0,\dots,N\}$, and set
$D=\operatorname{diag}\bigl((-1)^{x}\bigr)_{x\in\{0,\dots,N\}}$. 
For $d\in\ZZ_{\ge0}$ 
and $x,y\in\{0,\dots,N\}$ 
define
$$
\Khat^{(d)}_{n}(x,y) = (-1)^{x-y}\bigl(\delta_{xy}-\mathfrak{K}^{(d)}_{n}(x,y)\bigr)
=\bigl(D(\II-\mathfrak{K}^{(d)}_{n})D\bigr)(x,y).
$$
Similarly, for any kernel $K$ on the symmetric lattice 
$\mathfrak X=\{-N,\ldots,N\}$, we define the dual kernel $\widehat{K}(x,y)$ by
$$
\widehat{K}(x,y) = 
  (-1)^{x-y}\bigl(\delta_{xy}-K(x,y)\bigr),
\quad x,y\in \mathfrak X.
$$
\end{dfn}

\begin{remark}\label{rem:dual}
Since $D^{2}=\II$, the map $M\mapsto DMD$ is a similarity that preserves symmetry, idempotency and every principal minor; hence $\Khat^{(d)}_{n}$ is a symmetric projection kernel and defines the same determinantal hole process as the standard kernel
$\delta_{xy}-\mathfrak{K}^{(d)}_{n}$, with identical correlation functions. No property of $w$ beyond $\mathfrak{K}^{(d)}_{n}$ being a projection is used, so the construction applies to any weight discussed in this section. 
This signed complementary-kernel convention is the one used in the dual-ensemble construction of
\cite[Propositions~7.2 and~7.3]{baik2007discrete}.
\end{remark}

\begin{remark}\label{dual_kernels}
Proposition~\ref{As_prop_neq0} remains valid for the dual kernels when $v>0$.
More precisely, under the other assumptions of that proposition, suppose that, for every fixed \(j\in\ZZ\) and all fixed
\(A,B\in\mathbb X\),
\[
D_N\,
\widehat{\mathbb K}_{2n+j}
\bigl(x_N(A),x_N(B)\bigr)
\longrightarrow
\mathbb K_{\lim}(A,B).
\]
Then, for every fixed \(d\in\ZZ_{\geq0}\),
\[
D_N\,
\Khat^{(d)}_n
\bigl(x_N(A),x_N(B)\bigr)
\longrightarrow
\mathbb K_{\lim}(A,B).
\]

Indeed, for consecutive ranks,
\[
\widehat{\mathbb K}_{k+1}(x,y)
-
\widehat{\mathbb K}_{k}(x,y)
=
-(-1)^{x-y}p_k(x)p_k(y).
\]
In particular,
\[
p_k(x)^2
=
\widehat{\mathbb K}_{k}(x,x)
-
\widehat{\mathbb K}_{k+1}(x,x).
\]
Thus convergence of the dual kernels under fixed rank shifts gives the same pointwise estimate for the orthonormal functions as the one used in the proof of Proposition~\ref{As_prop_neq0}.

Since \(x_N(A)/N\to v\neq0\), the points \(x_N(A)\) and \(x_N(B)\)
are nonzero and have the same sign for all sufficiently large \(N\).
For such points,
\[
\Khat^{(0)}_n(x,y)
=
\widehat{\mathbb K}_{2n}(x,y)
+
\widehat{\mathbb K}_{2n}(-x,y),
\]
so the reflected term is controlled by the same
Christoffel--Darboux estimate as in the proof of
Proposition~\ref{As_prop_neq0}.

Finally, for \(x,y\neq0\), the definition of the dual kernels and the
Christoffel comparison identity give
\[
\frac{x}{y}\Khat^{(d+1)}_n(x,y)
-
\Khat^{(d)}_{n+1}(x,y)
=
-(-1)^{x-y}
\left[
\frac{x}{y}\mathfrak K^{(d+1)}_n(x,y)
-
\mathfrak K^{(d)}_{n+1}(x,y)
\right].
\]
Indeed, the additional diagonal term is
\[
(-1)^{x-y}
\left(\frac{x}{y}-1\right)\delta_{xy}=0.
\]
Hence the rank-one correction for the dual kernels has the same
absolute value as for the original kernels, and the proof of
Proposition~\ref{As_prop_neq0} applies without further changes.
\end{remark}

\section{Applications}
\label{sec:examples}
The purpose of this section is to apply Propositions~\ref{As_prop_0} and~\ref{As_prop_neq0} to two classical symmetric ensembles. We first treat the modified Krawtchouk kernel arising from the determinantal point process associated with skew $(\Sp_{2n},\Sp_{2k})$ Howe duality, as described in Section~\ref{sec:determ-point-proc}. The resulting kernel limits determine the local fluctuations of the corresponding random Young diagrams. We then apply the same transfer mechanism to the symmetric Hahn ensemble.

\subsection{Krawtchouk polynomials}
In this subsection we illustrate  Propositions ~\ref{As_prop_0}, \ref{As_prop_neq0} for the Krawtchouk polynomials. For the standard Krawtchouk weight we take the binomial distribution with parameter \(p=\tfrac12\), symmetric relative to $x=0$,
\begin{equation}\label{eq:krawtchouk_weight_centered}
w_{\mathrm{Kr}}^{(0)}(x)
=
\binom{2N}{N+x}\,2^{-2N},
\qquad x\in\mathfrak X=\{-N, \dots,N\}.
\end{equation}
Equivalently, if \(a_x=N+x\in\{0,1,\dots,2N\}\), then \(w_{\mathrm{Kr}}^{(0)}(x)\) is the usual Krawtchouk weight
\[
W^{\mathrm{Kr}}(a_x)=\binom{2N}{a_x}2^{-2N}
\]
written in symmetric coordinates. Since the weight \eqref{eq:krawtchouk_weight_centered} is even, the corresponding monic orthogonal polynomials have definite parity: \(P_{2\ell}(x)\) are even and \(P_{2\ell+1}(x)\) are odd.


The corresponding monic Krawtchouk polynomials can be written in hypergeometric form as (see \cite{Koekoek})
\begin{equation}\label{eq:krawtchouk_hypergeometric}
\widetilde K_m(x)
=
(-1)^mp^mm!\binom{2N}{m}\,
{}_2F_1\!\left(
-m,\,-N-x;\,-2N;\,\frac{1}{p}
\right),
\qquad p=\frac12.
\end{equation}
The associated orthogonality relation is
\begin{equation}\label{eq:krawtchouk_orthogonality}
\sum_{x\in\mathfrak X}
\widetilde K_\ell\!\left(x\right)
\widetilde K_m\!\left(x\right)
w_{\mathrm{Kr}}^{(0)}(x)
=
h_m^2\,\delta_{\ell m}, \quad h_m^2=(m!)^2\binom{2N}{m}2^{-2m},
\end{equation}
and the corresponding three-term recurrence has the form
\begin{equation}\label{eq:krawtchouk_recurrence}
x\,\widetilde K_m(x)
=
\widetilde K_{m+1}(x)+\beta_m\,\widetilde K_{m-1}(x),
\qquad
\beta_m=\frac{m(2N-m+1)}{4}.
\end{equation}
We denote by $\mathbb{K}_{n,\mathrm{Kr}}$ the Christoffel--Darboux kernel on $\mathfrak{X}$ and by \(\mathfrak K^{(0)}_{n,\mathrm{Kr}}(x,y)\) the Christoffel--Darboux kernel on the quadratic lattice, associated with \(w_{\mathrm{Kr}}^{(0)}(x)\).

Now we apply the iterated Christoffel transformation and define
\begin{equation}\label{eq:krawtchouk_christoffel}
w_{\mathrm{Kr}}^{(d)}(x)=x^{2d}w_{\mathrm{Kr}}^{(0)}(x).
\end{equation}
We denote by \(\mathfrak K^{(d)}_{n,\mathrm{Kr}}(x,y)\) the corresponding Christoffel--Darboux kernel on the quadratic lattice.

The transition from $w_{\mathrm{Kr}}^{(d)}(x)$ to
$w_{\mathrm{Kr}}^{(d+1)}(x)$ is governed by
Proposition~\ref{Main_Lemma}. The limiting regimes of
$\mathfrak K^{(d)}_{n,\mathrm{Kr}}(x,y)$ then follow from Propositions~\ref{As_prop_0}, \ref{As_prop_neq0} and the asymptotics of the ordinary Krawtchouk kernel,
as we show in this subsection.

We use the standard terminology for the limiting particle density: a band (bulk) is a region where the density lies strictly between $0$ and $1$, a void is a region where it is $0$, and a saturated region is one where it is $1$.

It is shown in~\cite{johansson2002non} that, in the
regime
\[
\frac nN\longrightarrow\gamma\in(0,1),
\]
the limiting band for the Krawtchouk kernel
\(\mathbb K_{2n,\mathrm{Kr}}\) consists of the macroscopic points
\(x/N\to u\) with
\[
u\in(-u_\ast(\gamma),u_\ast(\gamma)),
\qquad
u_\ast(\gamma)=2\sqrt{\gamma(1-\gamma)}.
\]
The regions outside the band are void when $\gamma<\frac12$ and saturated when $\gamma>\frac12$.
The local asymptotic behavior in the interior of the band is governed by the discrete sine kernel. If
\(\gamma\neq\frac12\), the two endpoints are soft edges, and the corresponding local asymptotics on the \(N^{1/3}\)-scale is governed by the Airy kernel. At the critical value \(\gamma=\frac12\), one has
\(u_\ast(\gamma)=1\), so the band reaches the endpoint of the lattice and the discrete Hermite regime arises; see Proposition~\ref{prop:kr_discr-Hermite}.

Propositions~\ref{prop:kr_bulk},\ref{prop:Airy_Kr},\ref{prop:kr_discr-Hermite} below show that, for every fixed $d\in\ZZ_{\ge0}$, the kernel $\mathfrak K^{(d)}_{n,\mathrm{Kr}}(x,y)$ has the same local asymptotic behavior away from the origin, in the appropriate particle or dual formulation (the dual kernel is introduced in Definition~\ref{def:dual}). At the origin, however, the reflected contribution survives; Proposition~\ref{prop:kr_origin} gives the resulting limits for $d=0$ and $d=1$.

Let $\mathbb K_{\sin,\; \phi}(A,B)$, $\phi\in [0,\pi]$,  denote the discrete sine kernel,
\begin{align*}
\mathbb K_{\sin,\; \phi}(A,B)= \begin{cases}
\dfrac{\sin\bigl(\phi(A-B)\bigr)}
      {\pi(A-B)},
& A\neq B,\\[2mm]
\dfrac{\phi}{\pi},
& A=B,
\end{cases}
\qquad A,B\in\ZZ.
\end{align*}
\begin{prop}[Bulk regime]\label{prop:kr_bulk}
Fix $\gamma\in(0,1)$ and set
\[
n=\lfloor\gamma N\rfloor,
\qquad
u_\ast(\gamma)=2\sqrt{\gamma(1-\gamma)}.
\]
Let $0<|u|<u_\ast(\gamma)$. Then, for every fixed \(d\in\ZZ_{\geq0}\) and all \(A,B\in\ZZ\),
\[
\lim_{N\to\infty}
\mathfrak K_{n,\mathrm{Kr}}^{(d)}
\bigl(\lfloor uN\rfloor+A,\lfloor uN\rfloor+B\bigr)
=
\mathbb K_{\sin,\;\phi_\gamma(u)}(A,B),
\]
where
\[
\phi_\gamma(u)
=
\arccos\left(
\frac{1-2\gamma}{\sqrt{1-u^2}}
\right).
\]
\end{prop}

\begin{proof}
Set
\[
X_N=\lfloor uN\rfloor,
\qquad
x_N(A)=X_N+A,
\qquad
A\in\ZZ.
\]
Since \(n=\lfloor\gamma N\rfloor\) with \(\gamma\in(0,1)\), we have
\(n\to\infty\) and \(N-n\to\infty\). Moreover,
\[
\frac{x_N(A)}{N}\longrightarrow u
\]
for every fixed \(A\in\ZZ\). Since
\(0<|u|<u_\ast(\gamma)\leq1\), the points \(x_N(A)\) belong to
\(\mathfrak X=\{-N,\ldots,N\}\) for all sufficiently large \(N\),
and the macroscopic localization condition in
Proposition~\ref{As_prop_neq0} is satisfied with
\[
D_N=1,
\qquad
v=u,
\qquad
\mathbb X=\ZZ.
\]

Suppose first that \(\gamma\leq\frac12\). 
For $A\neq B$, Lemma~2.8 of~\cite{johansson2002non}, after passing from the standard Krawtchouk coordinate to the centered coordinate used here, gives the required sine-kernel limit. For $A=B$, one uses the diagonal Christoffel--Darboux formula obtained from \cite[formula~(2.7)]{johansson2002non} by l'Hôpital's rule; the same saddle-point calculation as in the proof of Lemma~2.8 then gives
\[
\mathbb K_{2n+i,\mathrm{Kr}}
\bigl(x_N(A),x_N(A)\bigr)
\longrightarrow
\frac{\phi_\gamma(u)}{\pi}.
\]
\
Together with the fixed-rank-shift stability discussed in Remark~\ref{rem:fixed_shift}, this yields, for every fixed
\(i\in\ZZ\) and \(A,B\in\ZZ\),
\[
\lim_{N\to\infty}
\mathbb K_{2n+i,\mathrm{Kr}}
\bigl(x_N(A),x_N(B)\bigr)
=
\mathbb K_{\sin,\phi_\gamma(u)}(A,B),
\]
where
\[
\phi_\gamma(u)
=
\arccos\left(
\frac{1-2\gamma}{\sqrt{1-u^2}}
\right).
\]
Suppose now that $\gamma>\frac12$. Consider the signed complementary kernel
\[
\widehat{\mathbb K}_{2n+i,\mathrm{Kr}}(x,y)
=
(-1)^{x-y}
\left(
\delta_{xy}-\mathbb K_{2n+i,\mathrm{Kr}}(x,y)
\right).
\]
Section~3.2 of~\cite{baik2007discrete} identifies the hole ensemble of a discrete orthogonal polynomial ensemble with the ensemble associated with the dual weight defined in formula~(1.46). For the Krawtchouk family, Section~2.4.1 of~\cite{baik2007discrete} shows that this duality exchanges $p$ and $1-p$. Hence, for $p=\frac12$, the hole ensemble is again the same symmetric Krawtchouk ensemble, up to an irrelevant overall normalization of the weight. Propositions~7.2 and~7.3 of~\cite{baik2007discrete} identify its reproducing kernel, in the centered coordinates used here, with the signed complementary kernel above.

The rank of the hole ensemble is
\[
2N+1-(2n+i)
=
2(N-n)+(1-i),
\]
so its limiting particle ratio is
\(1-\gamma<\frac12\), with only a fixed shift of the projection rank.
Hence the preceding asymptotics give
\[
\widehat{\mathbb K}_{2n+i,\mathrm{Kr}}
\bigl(x_N(A),x_N(B)\bigr)
\longrightarrow
\mathbb K_{\sin,\phi_{1-\gamma}(u)}(A,B).
\]

Since $\phi_\gamma(u)=\pi-\phi_{1-\gamma}(u)$ and
\[
(-1)^{A-B}
\left(
\delta_{AB} -
\mathbb K_{\sin,\phi_{1-\gamma}(u)}(A,B)
\right) =
\mathbb K_{\sin,\phi_\gamma(u)}(A,B),
\]
we obtain the same limit for
\(\mathbb K_{2n+i,\mathrm{Kr}}\).

Thus, for every fixed \(i\in\ZZ\) and \(A,B\in\ZZ\),
\[
\mathbb K_{2n+i,\mathrm{Kr}}
\bigl(x_N(A),x_N(B)\bigr)
\longrightarrow
\mathbb K_{\sin,\phi_\gamma(u)}(A,B).
\]
All the assumptions of Proposition~\ref{As_prop_neq0} are therefore
satisfied. Applying that proposition gives, for every fixed
\(d\in\ZZ_{\geq0}\),
\[
\lim_{N\to\infty}
\mathfrak K_{n,\mathrm{Kr}}^{(d)}
\bigl(\lfloor uN\rfloor+A,\lfloor uN\rfloor+B\bigr)
=
\mathbb K_{\sin,\phi_\gamma(u)}(A,B),
\]
as required.
\end{proof}

Let $\mathbb{K}_{\mathrm{Ai}}(x,y)$  be the Airy kernel
\begin{align*}
\mathbb{K}_{\mathrm{Ai}}(x,y)=\int_{0}^{\infty}\operatorname{Ai}(x+t)\operatorname{Ai}(y+t)\,dt,
\end{align*}
where $\operatorname{Ai}(x)$ is the Airy function of the first kind.

In the next propositions we focus only on the right edge, the left edge can be recovered directly from the symmetry of the ensemble.
\begin{prop}[Soft-edge regime]\label{prop:Airy_Kr}
Fix
$
\gamma\in
\left(0,\frac12\right)\cup
\left(\frac12,1\right),
$
and set
$
n=\lfloor\gamma N\rfloor.
$
Define
\[
u_\ast(\gamma)=2\sqrt{\gamma(1-\gamma)},
\qquad
\xi_N=Nu_\ast(\gamma),
\]
and
\[
C_\gamma^{-1}
=
2^{2/3}
\gamma^{1/6}(1-\gamma)^{1/6}
|1-2\gamma|^{-2/3},
\qquad
D_N=C_\gamma N^{1/3}.
\]
For \(A\in\RR\), set
$
x_N(A)
=
\left\lfloor
\xi_N+D_NA
\right\rfloor.
$

If \(\gamma<\frac12\), then, for every fixed
\(d\in\ZZ_{\geq0}\) and all \(A,B\in\RR\),
\[
\lim_{N\to\infty}
D_N
\mathfrak K^{(d)}_{n,\mathrm{Kr}}
\bigl(x_N(A),x_N(B)\bigr)
=
\mathbb K_{\mathrm{Ai}}(A,B).
\]

If \(\gamma>\frac12\), then, for every fixed
\(d\in\ZZ_{\geq0}\) and all \(A,B\in\RR\),
\[
\lim_{N\to\infty}
D_N
\Khat^{(d)}_{n,\mathrm{Kr}}
\bigl(x_N(A),x_N(B)\bigr)
=
\mathbb K_{\mathrm{Ai}}(A,B).
\]
\end{prop}

\begin{proof}
We verify the assumptions of
Proposition~\ref{As_prop_neq0}.

Since $n=\lfloor\gamma N\rfloor$, we have $n\to\infty$ and $N-n\to\infty$. Set $\mathbb X=\RR$ and
$v=u_\ast(\gamma)$. Since $\gamma\neq\frac12$,
$u_\ast(\gamma)\in(0,1)$. Moreover, $D_N=o(N)$, and hence, for every fixed $A\in\RR$, $x_N(A)/N\to u_\ast(\gamma)$. In particular, $x_N(A)\in\mathfrak X$ for all sufficiently large $N$.

Since \(n=\lfloor\gamma N\rfloor\), for every fixed \(i\in\ZZ\),
\[
2n+i=2\gamma N+O(1).
\]
Thus a fixed shift of the rank changes the corresponding degree-to-size ratio only by $O(N^{-1})$. Since $\gamma\neq\frac12$,
the quantities $u_\ast(\gamma)$ and $C_\gamma$ are smooth in $\gamma$. Hence the resulting displacement of the soft-edge center is $O(1)$, while the Airy scale changes by a factor $1+o(1)$; these changes are negligible on the $N^{1/3}$-scale.

Suppose first that \(\gamma<\frac12\). The right band edge is then
adjacent to a void.  Set
\[
z_\gamma
=
\frac{\sqrt{1-\gamma}-\sqrt{\gamma}}
     {\sqrt{1-\gamma}+\sqrt{\gamma}}
\in(0,1).
\]
After specializing \cite[Theorem~1.5 and
formulas~(5.11), (5.13)]{betea2024} to the Krawtchouk ensemble,
replacing the particle number used there by \(2n\), and passing to
the centered lattice, one obtains
\[
\lim_{N\to\infty}
D_N
z_\gamma^{\,x_N(A)-x_N(B)}
\sqrt{
\frac{
w_{\mathrm{Kr}}^{(0)}(x_N(B))
}{
w_{\mathrm{Kr}}^{(0)}(x_N(A))
}
}
\,
\mathbb K_{2n,\mathrm{Kr}}
\bigl(x_N(A),x_N(B)\bigr)
=
\mathbb K_{\mathrm{Ai}}(A,B).
\]
The bounded displacement between the exact finite-\(N\) edge and
\(\xi_N\), caused by \(n-\gamma N=O(1)\), is negligible on the
\(N^{1/3}\)-scale.

The additional conjugating factor tends to one. Indeed,
\[
\frac{w_{\mathrm{Kr}}^{(0)}(x+1)}
     {w_{\mathrm{Kr}}^{(0)}(x)}
=
\frac{N-x}{N+x+1}.
\]
Uniformly for \(x=\xi_N+O(N^{1/3})\),
\[
\frac{N-x}{N+x+1}
=
\frac{1-u_\ast(\gamma)}
     {1+u_\ast(\gamma)}
\left(1+O(N^{-2/3})\right)
=
z_\gamma^2
\left(1+O(N^{-2/3})\right).
\]
Since
\[
|x_N(A)-x_N(B)|=O(N^{1/3}),
\]
it follows that
\[
\frac{
w_{\mathrm{Kr}}^{(0)}(x_N(A))
}{
w_{\mathrm{Kr}}^{(0)}(x_N(B))
}
=
z_\gamma^{\,2(x_N(A)-x_N(B))}
\bigl(1+o(1)\bigr).
\]
Consequently,
\[
z_\gamma^{\,x_N(A)-x_N(B)}
\sqrt{
\frac{
w_{\mathrm{Kr}}^{(0)}(x_N(B))
}{
w_{\mathrm{Kr}}^{(0)}(x_N(A))
}
}
=
1+o(1).
\]
The same argument applies after replacing $2n$ by $2n+i$ for any
fixed $i\in\ZZ$; see also Remark~\ref{rem:fixed_shift}. Therefore,
for all fixed $i\in\ZZ$ and $A,B\in\RR$,
\[
\lim_{N\to\infty}
D_N
\mathbb K_{2n+i,\mathrm{Kr}}
\bigl(x_N(A),x_N(B)\bigr)
=
\mathbb K_{\mathrm{Ai}}(A,B).
\]
Thus all assumptions of Proposition~\ref{As_prop_neq0} are satisfied
with $D_N=C_\gamma N^{1/3}$, $v=u_\ast(\gamma)$, and
$\mathbb X=\RR$.

Hence, for every fixed \(d\in\ZZ_{\geq0}\),
\[
\lim_{N\to\infty}
D_N
\mathfrak K^{(d)}_{n,\mathrm{Kr}}
\bigl(x_N(A),x_N(B)\bigr)
=
\mathbb K_{\mathrm{Ai}}(A,B).
\]

Suppose now that \(\gamma>\frac12\). In this case the right band edge is adjacent to a saturated region. By the particle--hole identification established in the proof of Proposition~\ref{prop:kr_bulk}, the signed complementary kernel
\[
\widehat{\mathbb K}_{2n+i,\mathrm{Kr}}(x,y)
=
(-1)^{x-y}
\left(
\delta_{xy}
-
\mathbb K_{2n+i,\mathrm{Kr}}(x,y)
\right)
\]
is the correlation kernel of the hole ensemble. Its rank is $2N+1-(2n+i)=2(N-n)+(1-i)$, so its limiting particle ratio is $1-\gamma<\frac12$, with only a fixed shift of the projection rank. Since
\[
u_\ast(1-\gamma)=u_\ast(\gamma),
\qquad
C_{1-\gamma}=C_\gamma,
\]
the void-edge Airy asymptotics
\cite[Lemma~7.16]{baik2007discrete} 
together with Remark~\ref{rem:fixed_shift}, give, for all fixed
$i\in\ZZ$ and $A,B\in\RR$,
\[
\lim_{N\to\infty}
D_N
\widehat{\mathbb K}_{2n+i,\mathrm{Kr}}
\bigl(x_N(A),x_N(B)\bigr)
=
\mathbb K_{\mathrm{Ai}}(A,B).
\]

Finally, Remark~\ref{dual_kernels} shows that
Proposition~\ref{As_prop_neq0} applies equally to the dual kernels. Hence, for every fixed \(d\in\ZZ_{\geq0}\),
\[
\lim_{N\to\infty}
D_N
\Khat^{(d)}_{n,\mathrm{Kr}}
\bigl(x_N(A),x_N(B)\bigr)
=
\mathbb K_{\mathrm{Ai}}(A,B).
\]
This completes the proof.
\end{proof}

Let $\{H_n\}_{n\geq 0}$ be the Hermite polynomials in the convention of~\cite[\S9.15]{Koekoek}. That is, $H_n(x)$ is the polynomial of degree $n$ with leading coefficient $2^n$ and  
$$
\int_{-\infty}^{+\infty} H_x(t)H_y(t)e^{-t^2}dt =
\sqrt{\pi}\, 2^{x}x!\,\delta_{xy}, \qquad x,y\in\ZZ_{\ge0}.
$$
Let $\mathbb K_{\mathrm{DHe^{+}}}^r(x,y)$, $r\in\RR$,  be the discrete Hermite kernel
(see \cite{Borodin_2017}):
\begin{align*}
\mathbb K_{\mathrm{DHe^{+}}}^r(x,y)=(\pi 2^{x+y}x!y!)^{-1/2}\int_{r}^{+\infty} H_x(t)H_y(t)e^{-t^2}dt, \quad x,y\in\ZZ_{\ge0}.
\end{align*}
\begin{prop}[Discrete Hermite regime]\label{prop:kr_discr-Hermite}
 Fix $r\in \mathbb{R}$ and suppose that
\begin{align*}
	 \frac{n}{N}=\frac{1}{2}-\frac{r}{2\sqrt{N}}+o(N^{-1/2}).
	 \end{align*}
Then for any fixed $d\in\ZZ_{\ge0}$ and
any $x,y\in\ZZ_{\ge0}$ we have
\[
\lim_{N\to\infty}
\mathfrak K_{n,\mathrm{Kr}}^{(d)}
\bigl(N-x,N-y\bigr)
=
\mathbb K_{\mathrm{DHe^{+}}}^{r}(x,y).
\]
\end{prop}
\begin{proof}
By \cite[Theorem~1.7 and Sections~4.4, 5.1]{betea2024},
the specialization $f\equiv g\equiv1$ gives the
$p=\frac12$ Krawtchouk ensemble, whose critical corner limit is
the discrete Hermite kernel.

To avoid confusion with the parameters $n,N$ used in the present
paper, when applying the notation of~\cite{betea2024} we write
$\widetilde n,\widetilde k$ for the parameters denoted there by
$n,k$, respectively. Fix $i\in\ZZ$ and set
\[
\widetilde n=2n+i,
\qquad
\widetilde k=2N+1-(2n+i).
\]
Then $\widetilde n+\widetilde k-1=2N$, so the corresponding
Krawtchouk ensemble in~\cite{betea2024} is precisely the ordinary
Krawtchouk ensemble on $\{0,\ldots,2N\}$ with projection rank
$2n+i$.

Moreover, the assumed scaling
\[
2n=N-r\sqrt N+o(\sqrt N)
\]
gives
\[
\frac{\widetilde k}{\widetilde n}\longrightarrow1,
\qquad
\frac{\widetilde k-\widetilde n}{\sqrt{\widetilde n}}
\longrightarrow2r.
\]
For $f\equiv g\equiv1$, the normalization in
\cite[Theorem~1.7]{betea2024} is $\tau=1/\sqrt2$. 
Define the finite-$N$ parameter by
\[
s_{N,i}
=
\tau\frac{\widetilde k-\widetilde n}
{\sqrt{\widetilde n}}.
\]
Then $s_{N,i}\longrightarrow\sqrt2\,r$ as $N\to\infty$, and the
fixed shift $i$ changes $s_{N,i}$ only by $O(N^{-1/2})$.
Thus the derivation in \cite[Section~4.4]{betea2024} applies unchanged
to this convergent sequence.

For the Krawtchouk site $N-x$, the corresponding uncentered coordinate is $2N-x$, and
\[
2N-x-\widetilde n+\frac12
=
\widetilde k-(x+1)+\frac12,
\qquad
\widetilde k
=
\widetilde n+
\frac{s_{N,i}}{\tau}\sqrt{\widetilde n}.
\]
Thus the indices in \cite[Theorem~1.7]{betea2024} are
$l=x+1$ and $l'=y+1$.

Passing from the correlation kernel used in~\cite{betea2024} to
the symmetric orthonormal Krawtchouk kernel by the same diagonal
conjugation as in the proof of Proposition~\ref{prop:Airy_Kr}, we
note that the corner normalization cancels:
\[
\sqrt{
\frac{w_{\mathrm{Kr}}^{(0)}(N-x)}
     {w_{\mathrm{Kr}}^{(0)}(N-y)}
}
\left(\frac{\tau}{\sqrt{\widetilde n}}\right)^{x-y}
\sqrt{\frac{x!}{y!}}
\longrightarrow1.
\]
Indeed, this follows from
\[
\frac{\binom{2N}{x}}{\binom{2N}{y}}
\sim
(2N)^{x-y}\frac{y!}{x!},
\qquad
\tau=\frac1{\sqrt2},
\qquad
\frac{\widetilde n}{N}\longrightarrow1.
\]
Using the normalization in
\cite[formulas~(4.35)--(4.38)]{betea2024} and
\[
\operatorname{He}_x(s)
=
2^{-x/2}H_x(s/\sqrt2),
\]
we therefore obtain, for every fixed $i\in\ZZ$ and
$x,y\in\ZZ_{\geq0}$,
\[
\lim_{N\to\infty}
\mathbb K_{2n+i,\mathrm{Kr}}(N-x,N-y)
=
\mathbb K_{\mathrm{DHe}^{+}}^r(x,y).
\]

We now apply Proposition~\ref{As_prop_neq0} with 
$X_N=N$, $v=1$, $D_N=1$, $\mathbb X=\ZZ_{\leq0}$, and
\[
\mathbb K_{\lim}(A,B)
=
\mathbb K_{\mathrm{DHe}^{+}}^r(-A,-B),
\qquad
A,B\in\ZZ_{\leq0}.
\]
For every fixed $A\in\ZZ_{\leq0}$, the point $N+A$ belongs to
the Krawtchouk lattice for all sufficiently large $N$, and
$(N+A)/N\to1$. The preceding ordinary-kernel limit holds for
every fixed rank shift $i\in\ZZ$, so all the assumptions of
Proposition~\ref{As_prop_neq0} are satisfied. Hence, for every
fixed $d\in\ZZ_{\geq0}$,
\[
\mathfrak K^{(d)}_{n,\mathrm{Kr}}(N+A,N+B)
\longrightarrow
\mathbb K_{\lim}(A,B).
\]
Taking $A=-x$ and $B=-y$ proves the proposition.
\end{proof}

\begin{remark}\label{rem:borodin-olshanski}
Theorem~6.4 of~Borodin--Olshanski~\cite{Borodin_2017} provides a secondary comparison with the preceding limit. To distinguish their notation from ours, let $N_{\mathrm{BO}}$ denote the number of particles, equivalently the projection rank, in their Krawtchouk ensemble, while $M$ is the upper endpoint of its lattice $\{0,\ldots,M\}$. For fixed $p\in(0,1)$, the critical scaling in that theorem should read
\[
\frac{N_{\mathrm{BO}}-pM}{\sqrt{pM}}
\longrightarrow
-\sqrt{2(1-p)}\,r,
\]
or, equivalently,
\[
\frac{N_{\mathrm{BO}}-pM}{\sqrt{2p(1-p)M}}
\longrightarrow-r.
\]
Indeed, the symmetric off-diagonal coefficient of the Krawtchouk difference operator contains the factor $\sqrt{p(1-p)}$; the displayed calculation in the proof of that theorem omits $\sqrt{1-p}$. At $p=\frac12$, the corrected scaling agrees with the parameter used above.

This correction is separate from the factor $\sqrt2$ in the formulation of Theorem~1.7 of~\cite{betea2024}. For $f\equiv g\equiv1$, the normalization there is $\tau=1/\sqrt2$, so the probabilists' Hermite parameter is $\sqrt2\,r$; the change of variables $u=\sqrt2\,t$ converts it to the physicists' Hermite parameter $r$.
\end{remark}

\begin{remark}
Borodin and Olshanski developed a spectral-projection method for proving local limits of discrete orthogonal polynomial ensembles; see \cite{borodin2007asymptotics,Borodin_2017}. In this approach, the Christoffel--Darboux kernel is realized as a spectral projection of the second-order difference operator $\mathcal D_N$ whose orthonormal polynomial functions $p_m$ are eigenfunctions. After centering and normalizing the operators, strong-resolvent convergence yields convergence of the corresponding spectral projections, provided that the spectral cutoff is not an eigenvalue of the limiting operator; the limiting kernel can then be identified through a generalized Fourier transform. This method was applied to the Krawtchouk to discrete Hermite transition in \cite[Theorem~6.4]{Borodin_2017}
and to the $q$-Krawtchouk bulk limit in \cite[Theorem~2]{NazNikSar22}. The direct spectral-projection interpretation of $\mathfrak K^{(0)}$ and $\mathfrak K^{(1)}$, together with the resulting preservation of local asymptotics away from the origin, is discussed in Appendix~\ref{sec:spectral_first_christoffel}.
\end{remark}

In the following regime, both the main term and the reflected term contribute to the limit, leading to a hard-wall modification of the discrete sine kernel.
\begin{prop}[Hard-wall sine kernel regime]\label{prop:kr_origin}
Suppose that
\[
\frac{n}{N}\longrightarrow\gamma\in(0,1),
\]
and set
\[
\phi_\gamma=\arccos(1-2\gamma).
\]
Then, for every fixed \(A,B\in\ZZ\),
\begin{align*}
\lim_{N\to\infty}
\mathfrak K_{n,\mathrm{Kr}}^{(0)}(A,B)
&=
\eta(A)\eta(B)
\left[
\mathbb K_{\sin,\phi_\gamma}(A,B)
+
\mathbb K_{\sin,\phi_\gamma}(-A,B)
\right],\\
\lim_{N\to\infty}
\mathfrak K_{n,\mathrm{Kr}}^{(1)}(A,B)
&=
\mathbb K_{\sin,\phi_\gamma}(A,B)
-
\mathbb K_{\sin,\phi_\gamma}(-A,B).
\end{align*}
\end{prop}

\begin{proof}
After shifting the Krawtchouk lattice from $\{0,\ldots,2N\}$ to $\{-N,\ldots,N\}$, 
the ordinary-kernel argument in the proof of Proposition~\ref{prop:kr_bulk} applies also at $u=0$, which is an interior band point for every $\gamma\in(0,1)$. The restriction $u\neq0$ in that proposition is needed only for the subsequent application of Proposition~\ref{As_prop_neq0}. Thus, for $\gamma\leq\frac12$, 
\[
\lim_{N\to\infty}
\mathbb K_{2n,\mathrm{Kr}}(A,B)
=
\mathbb K_{\sin,\phi_\gamma}(A,B),
\qquad
A,B\in\ZZ.
\]

For $\gamma>\frac12$, we again appeal to the proof of Proposition~\ref{prop:kr_bulk}, now to its particle--hole argument, which gives the same conclusion using $\phi_\gamma=\pi-\phi_{1-\gamma}$. Thus part~(1) of
Proposition~\ref{As_prop_0} applies with
\[
\mathbb K_{\lim}=\mathbb K_{\sin,\phi_\gamma},
\]
and gives both asserted limits.
\end{proof}

\begin{remark}
The discrete hard-wall sine kernel has previously appeared in asymptotic representation theory in the work of A.~Borodin and J.~Kuan \cite{borodin2010random}.
It was also studied in connection with XX chains~\cite{Fagotti2011}. A continuous analogue of the discrete hard-wall sine kernel is well-known~\cite{Erhardt2007Dyson,verbaarschot1994,verbaarschot2000} and can be found in chiral ensembles in random matrix theory~\cite[Sections 3.1 and 7.2.7]{forrester2010log}. 
\end{remark}
According to Proposition~\ref{As_prop_neq0}, the kernels $\mathfrak K_{n,\mathrm{Kr}}^{(0)}, \mathfrak K_{n,\mathrm{Kr}}^{(1)}, \mathfrak K_{n,\mathrm{Kr}}^{(2)},\dots$ are close to one another away from a neighborhood of $0$. This is illustrated in Figure~\ref{fig:symplectic}, where we plot the corresponding densities for finite but relatively large values of $N$. 
The densities display large oscillations near the left edge. In this region, Proposition~\ref{prop:kr_origin} shows that the asymptotic behavior is governed by the discrete hard-wall sine kernel, equivalently  the odd-reflection sine kernel, for $\mathfrak K_{n,\mathrm{Kr}}^{(1)}$, and by the modified even-reflection sine kernel 
$$
\eta(A)\eta(B)\left[
\mathbb K_{\sin,\phi}(A,B)+\mathbb K_{\sin,\phi}(-A,B)\right]
$$ 
for $\mathfrak K_{n,\mathrm{Kr}}^{(0)}$. More generally, we expect a finite-rank modification of the even-reflection sine kernel 
$$
\mathbb K_{\sin,\phi}(A,B)+\mathbb K_{\sin,\phi}(-A,B)
$$ 
to arise for the even-indexed kernels $\mathfrak K_{n,\mathrm{Kr}}^{(2d)}$, $d> 0$ and, respectively,  a finite-rank modification of the odd-reflection sine kernel for the odd-indexed kernels $\mathfrak K_{n,\mathrm{Kr}}^{(2d+1)}$, $d>0$.

\begin{figure}[h!tb]
  \centering
  \includegraphics[width=0.49\linewidth]{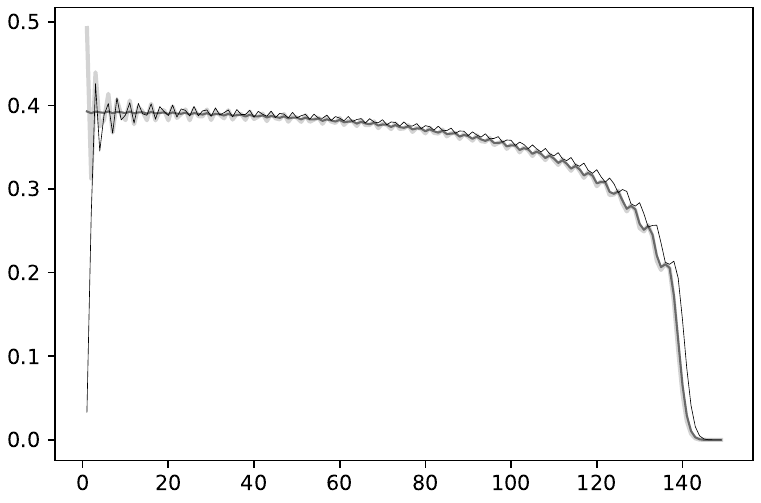}
  \includegraphics[width=0.49\linewidth]{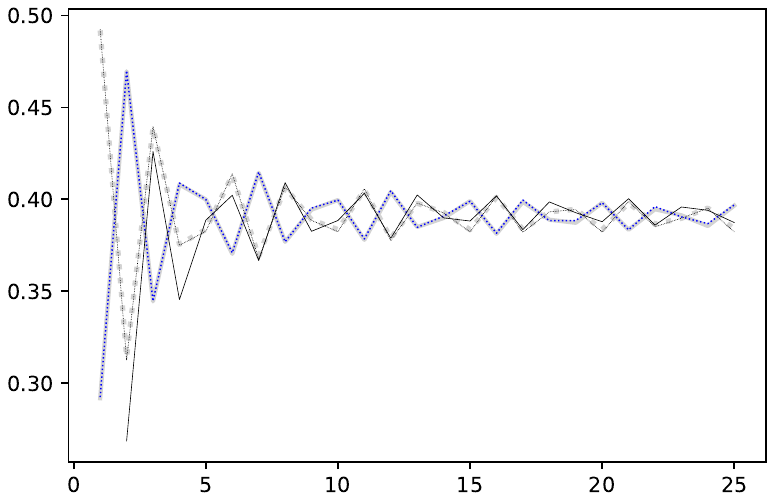}
  \caption{{\it Left:} Plot of the densities $\mathfrak K_{n,\mathrm{Kr}}^{(0)}(x,x)$ {\it(thick light gray line)},  $\mathbb{K}_{2n,\mathrm{Kr}}(x,x)$ {\it(dark gray line)}, $\mathfrak K_{n,\mathrm{Kr}}^{(2)}(x,x)$ {\it(thin black line)} for $n=50, N=150$. \newline
{\it Right:} Plot of the densities $\mathfrak K_{n,\mathrm{Kr}}^{(1)}(x,x)$ {\it(dotted blue line)}, $\mathbb K_{\sin,\phi}(x,x)-\mathbb K_{\sin,\phi}(-x,x)$ {\it(thick solid light gray line)}; $\mathfrak{K}^{(0)}_{n,\mathrm{Kr}}(x,x)$ {\it(thin dotted black line)}, $\eta^2(x)[\mathbb K_{\sin,\phi}(x,x)+\mathbb K_{\sin,\phi}(-x,x)]$ {\it(thick dotted light gray line)}; $\mathfrak K_{n,\mathrm{Kr}}^{(2)}(x,x)$ {\it(thin black line)}, for $n=50, N=150$ near the left edge. }
  \label{fig:symplectic}
\end{figure}
Combining Propositions~\ref{prop:kr_bulk}, \ref{prop:Airy_Kr}, \ref{prop:kr_discr-Hermite} and \ref{prop:kr_origin}, we obtain the complete list of local fluctuation regimes for the symplectic Schur measure~\eqref{eq:young-measure}: the discrete sine kernel in the bulk, the Airy kernel at the right edge, the discrete hard-wall sine kernel at the left corner, and the discrete Hermite
 kernel in the critical regime
$2n=N-r\sqrt N+o(\sqrt N)$.

\subsection{Hahn polynomials}

In this subsection we illustrate Propositions~\ref{As_prop_0}, \ref{As_prop_neq0} for the Hahn polynomials.
Since the whole method of Section~\ref{sec:outline-method} is formulated for symmetric
weights, we restrict ourselves to the symmetric Hahn family.

For \(\zeta>-1\), the standard Hahn weight on \(\{0,1,\dots,2N\}\) is
\begin{equation}\label{eq:hahn_weight_standard}
W_{\mathrm{Hn}}(x)
=
\binom{\zeta+x}{x}\binom{\zeta+2N-x}{2N-x},
\qquad x=0,1,\dots,2N,
\end{equation}
see \cite[Ch.9.5]{Koekoek}.
Then we obtain the symmetric weight
\begin{equation}\label{eq:hahn_weight_symmetric}
w_{\mathrm{Hn}}^{(0)}(x)
:=
W_{\mathrm{Hn}}(N+x)
=
\binom{\zeta+N+x}{N+x}\binom{\zeta+N-x}{N-x},
\qquad x\in\mathfrak X=\{-N,...,N\}.
\end{equation}

Let \(\{\widetilde{Q}_m\}_{m=0}^{2N}\) be the monic orthogonal polynomials on \(\mathfrak X\)
with respect to \(w_{\mathrm{Hn}}^{(0)}\),
which satisfy the orthogonality relation
\begin{equation}\label{eq:Hanh_orthogonality}
\sum_{x\in\mathfrak X}
\widetilde{Q}_\ell\!\left(x\right)
\widetilde{Q}_m\!\left(x\right)
w_{\mathrm{Hn}}^{(0)}(x)
=
h_m^2\,\delta_{\ell m}, \quad h_m^2=\frac{m!\,(\zeta+1)_m^2\,(2m+2\zeta+2)_{2N-m}}
{(2N-m)!\,(m+2\zeta+1)_m},
\end{equation} 
with three-term recurrence relation in the symmetric form (see \cite{Koekoek})
\begin{equation}\label{eq:hahn_ttr_symmetric}
x \widetilde{Q}_m(x)=\widetilde{Q}_{m+1}(x)+\beta_m \widetilde{Q}_{m-1}(x),
\qquad 
\beta_m=\frac{m(m+2\zeta)(2N-m+1)(2N+m+2\zeta+1)}
{4(2m+2\zeta-1)(2m+2\zeta+1)}.
\end{equation}
We denote by $\mathbb K_{n,\mathrm{Hn}}$ the Christoffel--Darboux kernel on $\mathfrak X$ associated with $w_{\mathrm{Hn}}^{(0)}$.

Now we apply the iterated Christoffel transformation and set
\[
w_{\mathrm{Hn}}^{(d)}(x) = x^{2d}w_{\mathrm{Hn}}^{(0)}(x),
\qquad
d\in\ZZ_{\geq0}.
\]
We denote by $\mathfrak K^{(d)}_{n,\mathrm{Hn}}(x,y)$ the corresponding Christoffel--Darboux kernel on the quadratic lattice. As in the Krawtchouk case, Proposition~\ref{Main_Lemma} relates successive values of $d$, while Propositions~\ref{As_prop_0} and~\ref{As_prop_neq0} transfer the known local asymptotics of the ordinary Hahn kernel to these Christoffel-transformed kernels.

We first consider the proportional regime
\[
\frac nN\longrightarrow\gamma\in(0,1),
\qquad
\frac{\zeta}{N}\longrightarrow\kappa>0.
\]
In this regime the limiting particle density has a single symmetric band $[-u_{\ast}(\gamma,\kappa),u_{\ast}(\gamma,\kappa)]$. There is a critical value $\gamma_{\mathrm c}(\kappa)$ separating two possible phase configurations. For $\gamma<\gamma_{\mathrm c}(\kappa)$ the regions outside the band are void, whereas for $\gamma>\gamma_{\mathrm c}(\kappa)$ they are saturated. At the critical value the band reaches the endpoints of the lattice:
\[
u_{\ast}(\gamma_{\mathrm c}(\kappa),\kappa)=1.
\]
The corresponding critical soft-edge regime is not considered below. The explicit density, band endpoint, critical value, and soft-edge normalization are collected in Lemma~\ref{lem:Hahn_edge_parameters}.

In the interior of the band, away from the origin, the local asymptotics are governed by the discrete sine kernel. If $\gamma\neq\gamma_{\mathrm c}(\kappa)$, the two band endpoints are generic soft edges and the local asymptotics on the $N^{1/3}$-scale are governed by the Airy kernel. In the void--band--void regime the Airy limit is formulated for the particle kernel, while in the saturated--band--saturated regime it is naturally formulated for the dual kernel. By symmetry it is sufficient to treat the right edge.

As in the Krawtchouk case, every fixed Christoffel transform has the same local asymptotic behavior away from the origin. At the origin, however, the reflected contribution survives and produces the hard-wall sine kernels described below.

Finally, we also consider a different degeneration of the Hahn parameters, with $\zeta$ fixed and $2n^2/N\to r>0$. Near the endpoint of the lattice this gives the discrete Laguerre kernel.
\begin{lemma}[Hahn particle density and edge parameters]
\label{lem:Hahn_edge_parameters}
Fix $\gamma\in(0,1)$, $\kappa>0$,  and define
\[
\Delta_{\gamma,\kappa} =
\sqrt{
\gamma(1-\gamma)
(\kappa+\gamma)
(\kappa+\gamma+1)
},
\]
\[
q_{\gamma,\kappa}
=
\kappa-2\gamma(\kappa+\gamma),
\qquad
r_{\gamma,\kappa}
=
\kappa(\kappa+1)+2\gamma(\kappa+\gamma).
\]
The right endpoint of the limiting band in the \(x/N\)-coordinate is
\[
u_{\ast}(\gamma,\kappa)
=
\frac{2\Delta_{\gamma,\kappa}}{\kappa+2\gamma}.
\]
For \(\lvert u\rvert\leq u_{\ast}(\gamma,\kappa)\), set
\[
Z_{\gamma,\kappa}(u)
=
(\kappa+2\gamma)
\sqrt{u_{\ast}(\gamma,\kappa)^2-u^2}.
\]

The limiting particle density
\(\rho_{\gamma,\kappa}\colon[-1,1]\to[0,1]\) is given as follows.
If \(q_{\gamma,\kappa}>0\), then
\[
\rho_{\gamma,\kappa}(u)
=
\begin{cases}
\displaystyle
\frac{1}{\pi}
\left[
\arctan\left(
\frac{Z_{\gamma,\kappa}(u)}
{q_{\gamma,\kappa}}
\right)
-
\arctan\left(
\frac{Z_{\gamma,\kappa}(u)}
{r_{\gamma,\kappa}}
\right)
\right],
&
\lvert u\rvert\leq u_{\ast}(\gamma,\kappa),
\\[1.4em]
0,
&
u_{\ast}(\gamma,\kappa)<\lvert u\rvert\leq 1.
\end{cases}
\]
If \(q_{\gamma,\kappa}<0\), then
\[
\rho_{\gamma,\kappa}(u)
=
\begin{cases}
\displaystyle
1-
\frac{1}{\pi}
\left[
\arctan\left(
\frac{Z_{\gamma,\kappa}(u)}
{-q_{\gamma,\kappa}}
\right)
+
\arctan\left(
\frac{Z_{\gamma,\kappa}(u)}
{r_{\gamma,\kappa}}
\right)
\right],
&
\lvert u\rvert\leq u_{\ast}(\gamma,\kappa),
\\[1.4em]
1,
&
u_{\ast}(\gamma,\kappa)<\lvert u\rvert\leq 1.
\end{cases}
\]
At the critical value \(q_{\gamma,\kappa}=0\), one has
\(u_{\ast}(\gamma,\kappa)=1\), and the limiting density is
\[
\rho_{\gamma,\kappa}(u)
=
\frac{1}{2}
-
\frac{1}{\pi}
\arctan\left(
\frac{
(\kappa+2\gamma)\sqrt{1-u^2}
}{
r_{\gamma,\kappa}
}
\right),
\qquad
\lvert u\rvert<1.
\]

The critical value separating the void and saturated regimes is
\[
\gamma_{\mathrm c}(\kappa)
=
\frac{\sqrt{\kappa^2+2\kappa}-\kappa}{2}.
\]
More precisely,
\[
q_{\gamma,\kappa}>0
\quad\Longleftrightarrow\quad
\gamma<\gamma_{\mathrm c}(\kappa),
\]
in which case the limiting configuration is
void--band--void, whereas
\[
q_{\gamma,\kappa}<0
\quad\Longleftrightarrow\quad
\gamma>\gamma_{\mathrm c}(\kappa),
\]
in which case the limiting configuration is
saturated--band--saturated. Moreover,
\[
u_{\ast}(\gamma,\kappa)<1
\quad\Longleftrightarrow\quad
\gamma\neq\gamma_{\mathrm c}(\kappa),
\]
while
\[
u_{\ast}(\gamma,\kappa)=1
\]
at the critical value.

For \(\gamma\neq\gamma_{\mathrm c}(\kappa)\), the corresponding
normalization constant in the \(N^{1/3}\) soft-edge scaling is
\[
C_{\gamma,\kappa}^{-1}
=
\frac{
2^{2/3}
\Delta_{\gamma,\kappa}^{1/3}
(\kappa+2\gamma)^{5/3}
}{
\bigl(
\lvert q_{\gamma,\kappa}\rvert
r_{\gamma,\kappa}
\bigr)^{2/3}
}.
\]
In the void-edge regime this is the normalization for the particle
kernel, while in the saturated-edge regime it is the normalization
for the dual kernel.
\end{lemma}

\begin{proof}
Under the parameter identification
\[
M=2N+1,
\qquad
A=B=\frac{\kappa}{2},
\qquad
c=\gamma,
\qquad
t=\frac12+\frac{x}{M},
\]
the formulas for the limiting particle density, the band endpoints,
and the void/saturated phase classification follow from
\cite[formulas~(2.73), (2.75)--(2.79) and
Theorem~2.17]{baik2007discrete}. In particular, the right endpoint
$\beta$ in the $t$-coordinate becomes $2\beta-1=u_{\ast}(\gamma,\kappa)$
in the $x/N$-coordinate.

Expanding the densities in \cite[formulas~(2.77) and
(2.79)]{baik2007discrete} at \(t=\beta\), and then using the
edge coefficients defined in \cite[formulas~(3.16) and
(3.18)]{baik2007discrete} together with
\cite[Theorems~3.7 and~3.8]{baik2007discrete}, gives the stated
soft-edge normalization constants. 
The details of these substitutions and simplifications are given in
Appendix~\ref{app:Hahn_edge_parameters}.
\end{proof}

\begin{prop}[Bulk regime]\label{prop:hahn_bulk}
Fix \(\kappa>0\) and
$\gamma\in
\left(0,\gamma_{\mathrm c}(\kappa)\right)\cup
\left(\gamma_{\mathrm c}(\kappa),1\right)$,
and set
\[
n=\lfloor\gamma N\rfloor,
\qquad
\zeta=\kappa N.
\]
Let \(u_{\ast}(\gamma,\kappa)\) and
\(\rho_{\gamma,\kappa}\) be as in
Lemma~\ref{lem:Hahn_edge_parameters}. Suppose that
$
0<|u|<u_{\ast}(\gamma,\kappa),
$
and set
\[
\phi_{\gamma,\kappa}(u)
=
\pi\rho_{\gamma,\kappa}(u)
\in(0,\pi).
\]
Then, for every fixed \(d\in\ZZ_{\geq0}\) and all
\(A,B\in\ZZ\),
\[
\lim_{N\to\infty}
\mathfrak K^{(d)}_{n,\mathrm{Hn}}
\bigl(
\lfloor uN\rfloor+A,
\lfloor uN\rfloor+B
\bigr)
=
\mathbb K_{\sin,\phi_{\gamma,\kappa}(u)}(A,B).
\]
\end{prop} 
\begin{proof}
We verify the assumptions of
Proposition~\ref{As_prop_neq0}. Set
$X_N=\lfloor uN\rfloor$, $\mathbb X=\ZZ$. Then
${X_N}/{N}\longrightarrow u$. By Lemma~\ref{lem:Hahn_edge_parameters},
\(u_{\ast}(\gamma,\kappa)\leq1\). Hence
$
u\in(-1,1)\setminus\{0\},
$
and, for every fixed \(A\in\ZZ\),
\[
X_N+A\in\mathfrak X=\{-N,\ldots,N\}
\]
for all sufficiently large \(N\).

To apply the asymptotic results for the Hahn polynomials, put $M=2N+1$.
Since $n=\lfloor\gamma N\rfloor$ and $\zeta=\kappa N$, for every fixed $i\in\ZZ$ we have $2n+i=\gamma M+O(1)$ and $\zeta=\frac{\kappa}{2}M+O(1)$. 
Finally, the bulk-kernel asymptotics
\cite[Lemma~7.13]{baik2007discrete}, specialized to the Hahn weight via \cite[\S2.4.2 and Theorem~2.17]{baik2007discrete} together with the fixed-shift observation in Remark~\ref{rem:fixed_shift}, give, for all fixed \(i,A,B\in\ZZ\),
\[
\lim_{N\to\infty}
\mathbb K_{2n+i,\mathrm{Hn}}
(X_N+A,X_N+B)
=
\mathbb K_{\sin,\phi_{\gamma,\kappa}(u)}(A,B).
\]
See also~\cite[Theorem~1]{Gorin2008} for the corresponding Hahn bulk limit.

Here the phase is the local occupation density multiplied by
\(\pi\); by Lemma~\ref{lem:Hahn_edge_parameters}, it is precisely
\[
\phi_{\gamma,\kappa}(u)
=
\pi\rho_{\gamma,\kappa}(u).
\]

All assumptions of Proposition~\ref{As_prop_neq0} are
therefore satisfied with $D_N=1$, $x_N(A) = X_N+A$, $v=u$, and $\mathbb X=\ZZ$.
Applying that proposition completes the proof.
\end{proof}
\begin{remark}\label{rem:Hahn_critical_bulk}
The exclusion of $\gamma=\gamma_{\mathrm c}(\kappa)$ is due to the formal scope of the asymptotic results in~\cite{baik2007discrete}, rather than to a change in the local bulk behavior. At the critical value the band fills the entire macroscopic interval, so every fixed point with $|u|<1$ remains an interior band point. However, one of the endpoint assumptions in \cite[\S2.1.2]{baik2007discrete} fails, and the bulk-universality results of that source, including Lemma~7.13, are stated under those assumptions.

The corresponding sine-kernel limit can nevertheless be recovered directly, for example, by bracketing the critical-rank Hahn projection between the noncritical projections associated with $\gamma_{\mathrm c}(\kappa)-\varepsilon$ and $\gamma_{\mathrm c}(\kappa)+\varepsilon$, and then using continuity of the limiting density together with the Cauchy--Schwarz inequality for the differences of nested projections. For integer parameter sequences such as $\zeta_N=\lfloor\kappa N\rceil$, the ordinary Hahn bulk limit also follows from the specialization of \cite[Theorem~1]{Gorin2008}, after translating its parameters to the symmetric Hahn ensemble. 
\end{remark}

We next turn to the noncritical soft edges. By symmetry it is enough to consider the right edge.

\begin{prop}[Soft-edge regime]
\label{prop:Airy_Hn}
Fix $\kappa>0$ and
$\gamma\in
\left(0,\gamma_{\mathrm c}(\kappa)\right)\cup
\left(\gamma_{\mathrm c}(\kappa),1\right)$,
and set
\[
n=\lfloor\gamma N\rfloor, \quad \zeta=\kappa N.
\]
Let \(u_{\ast}=u_{\ast}(\gamma,\kappa)<1\) denote the right soft edge in the \(x/N\)-coordinate, and let
\(C_{\gamma,\kappa}>0\) be the corresponding edge-normalization constant. For \(A,B\in\RR\), set
\[
x_N(A)
=
\left\lfloor
u_{\ast}N+C_{\gamma,\kappa}A N^{1/3}
\right\rfloor.
\]

If the right soft edge is adjacent to a void region, then, for any fixed \(d\in\ZZ_{\geq0}\),
\[
\lim_{N\to\infty}
C_{\gamma,\kappa}N^{1/3}
\mathfrak K^{(d)}_{n,\mathrm{Hn}}
\bigl(x_N(A),x_N(B)\bigr)
=
\mathbb K_{\mathrm{Ai}}(A,B).
\]

If the right soft edge is adjacent to a saturated region, then, for any fixed \(d\in\ZZ_{\geq0}\),
\[
\lim_{N\to\infty}
C_{\gamma,\kappa}N^{1/3}
\Khat^{(d)}_{n,\mathrm{Hn}}
\bigl(x_N(A), x_N(B)\bigr)
=
\mathbb K_{\mathrm{Ai}}(A,B).
\]
\end{prop}

\begin{proof}
Set
\[
D_N=C_{\gamma,\kappa}N^{1/3},
\qquad
\xi_N=u_{\ast}(\gamma,\kappa)N,
\qquad
x_N(A)=\lfloor \xi_N+D_NA\rfloor,
\qquad A\in\RR.
\]
Since \(n=\lfloor\gamma N\rfloor\) with \(\gamma\in(0,1)\), we have
\(n\to\infty\) and \(N-n\to\infty\). Moreover, by
Lemma~\ref{lem:Hahn_edge_parameters}, 
$0<u_{\ast}(\gamma,\kappa)<1$
whenever \(\gamma\neq\gamma_{\mathrm c}(\kappa)\). Since \(D_N=o(N)\), it follows that, for every fixed \(A\in\RR\),
\[
\frac{x_N(A)}{N}\longrightarrow u_{\ast}(\gamma,\kappa),
\]
and \(x_N(A)\in\mathfrak X=\{-N,\ldots,N\}\) for all sufficiently
large \(N\).

As in the proof of Proposition~\ref{prop:hahn_bulk}, put $M=2N+1$.
For every fixed $i\in\ZZ$,
\[
2n+i=\gamma M+O(1),
\qquad
\zeta=\frac{\kappa}{2}M+O(1).
\]
Thus the bounded-shift conditions in the asymptotic framework
of~\cite{baik2007discrete} are satisfied; see also
Remark~\ref{rem:fixed_shift}.

Suppose first that
\(\gamma<\gamma_{\mathrm c}(\kappa)\). By
Lemma~\ref{lem:Hahn_edge_parameters}, the right edge is then adjacent
to a void. The Airy-kernel asymptotics
\cite[Theorem~3.7 and Lemma~7.16]{baik2007discrete}, specialized to
the Hahn weight as in
\cite[\S2.4.2 and Theorem~2.17]{baik2007discrete}, together with the
normalization computed in
Lemma~\ref{lem:Hahn_edge_parameters} and
Appendix~\ref{app:Hahn_edge_parameters}, give, for every fixed
\(i\in\ZZ\) and \(A,B\in\RR\),
\[
\lim_{N\to\infty}
D_N\,
\mathbb K_{2n+i,\mathrm{Hn}}
\bigl(x_N(A),x_N(B)\bigr)
=
\mathbb K_{\mathrm{Ai}}(A,B).
\]
Hence all the assumptions of Proposition~\ref{As_prop_neq0} are
satisfied with
\[
\mathbb X=\RR,
\qquad
v=u_{\ast}(\gamma,\kappa).
\]
That proposition therefore yields, for every fixed
\(d\in\ZZ_{\geq0}\),
\[
\lim_{N\to\infty}
D_N\,
\mathfrak K^{(d)}_{n,\mathrm{Hn}}
\bigl(x_N(A),x_N(B)\bigr)
=
\mathbb K_{\mathrm{Ai}}(A,B).
\]

Suppose now that
\(\gamma>\gamma_{\mathrm c}(\kappa)\). The right edge is then adjacent
to a saturated region. Introduce the signed complementary kernel
\[
\widehat{\mathbb K}_{2n+i,\mathrm{Hn}}(x,y)
=
(-1)^{x-y}
\left(
\delta_{xy}-\mathbb K_{2n+i,\mathrm{Hn}}(x,y)
\right).
\]
Applying \cite[Lemma~7.16]{baik2007discrete} to the dual ensemble and using \cite[Propositions~7.2 and~7.3]{baik2007discrete} (equivalently, the duality argument in the proof of \cite[Theorem~3.8]{baik2007discrete}), with the same parameter identification and normalization as above, gives, for every
fixed \(i\in\ZZ\) and \(A,B\in\RR\),
\[
\lim_{N\to\infty}
D_N\,
\widehat{\mathbb K}_{2n+i,\mathrm{Hn}}
\bigl(x_N(A),x_N(B)\bigr)
=
\mathbb K_{\mathrm{Ai}}(A,B).
\]
By Remark~\ref{dual_kernels},
Proposition~\ref{As_prop_neq0} applies equally to the dual kernels.
Consequently, for every fixed \(d\in\ZZ_{\geq0}\),
\[
\lim_{N\to\infty}
D_N\,
\Khat^{(d)}_{n,\mathrm{Hn}}
\bigl(x_N(A),x_N(B)\bigr)
=
\mathbb K_{\mathrm{Ai}}(A,B).
\]
This proves both assertions.
\end{proof}

\begin{prop}[Hard-wall sine kernel regime]\label{Hahn_hard-wall_sine}
Fix $\kappa>0$ and
$\gamma\in
\left(0,\gamma_{\mathrm c}(\kappa)\right)\cup
\left(\gamma_{\mathrm c}(\kappa),1\right)$,
and set
$
n=\lfloor\gamma N\rfloor$, $\zeta=\kappa N$,
and
$
\phi_{\gamma,\kappa}
=
\pi\rho_{\gamma,\kappa}(0),
$
where \(\rho_{\gamma,\kappa}\) is the limiting particle density from
Lemma~\ref{lem:Hahn_edge_parameters}. Then, for every fixed
\(A,B\in\ZZ\),
\begin{align*}
\lim_{N\to\infty}
\mathfrak K^{(0)}_{n,\mathrm{Hn}}(A,B)
&=
\eta(A)\eta(B)
\left[
\mathbb K_{\sin,\phi_{\gamma,\kappa}}(A,B)
+
\mathbb K_{\sin,\phi_{\gamma,\kappa}}(-A,B)
\right],\\
\lim_{N\to\infty}
\mathfrak K^{(1)}_{n,\mathrm{Hn}}(A,B)
&=
\mathbb K_{\sin,\phi_{\gamma,\kappa}}(A,B)
-
\mathbb K_{\sin,\phi_{\gamma,\kappa}}(-A,B).
\end{align*}
\end{prop}
\begin{proof}
The ordinary-kernel argument in the proof of Proposition~\ref{prop:hahn_bulk} applies also at $u=0$, which is an interior band point. The restriction $u\neq0$ in that proposition is needed only for the subsequent application of Proposition~\ref{As_prop_neq0}. Thus, for all fixed $A,B\in\ZZ$,
\[
\lim_{N\to\infty}
\mathbb K_{2n,\mathrm{Hn}}(A,B)
=
\mathbb K_{\sin,\phi_{\gamma,\kappa}}(A,B),
\qquad
\phi_{\gamma,\kappa}
=
\pi\rho_{\gamma,\kappa}(0).
\]
Consequently, part~(1) of Proposition~\ref{As_prop_0} applies with
$
\mathbb K_{\lim}
=
\mathbb K_{\sin,\phi_{\gamma,\kappa}},
$
and gives both asserted limits.
\end{proof}
We conclude the Hahn discussion with the fixed-$\zeta$ discrete Laguerre regime.

Let $\mathbb K_{\mathrm{DLaguerre}^{-}(r;\beta)}(x,y)$, $\beta>0$, $r>0$, be the discrete Laguerre kernel
(see \cite{Borodin_2017}):
\[
\mathbb K_{\mathrm{DLaguerre}^{-}(r;\beta)}(x,y)
=
\left(
\frac{x!\,y!}
{\Gamma(x+\beta)\Gamma(y+\beta)}
\right)^{1/2}
\int_0^r
L_x^{(\beta-1)}(t)
L_y^{(\beta-1)}(t)
t^{\beta-1}e^{-t}\,dt,
\qquad
x,y\in\ZZ_{\geq0},
\]
where \(L_x^{(\beta-1)}\) denotes the generalized Laguerre polynomial in the standard normalization.

\begin{prop}[Discrete Laguerre regime]\label{prop:hahn_discr-Laguerre}
Fix \(\zeta>-1\) and suppose that
\[
\frac{2n^2}{N}\longrightarrow r>0.
\]
Then for any fixed $d\in\ZZ_{\ge0}$ and
any $x,y\in\ZZ_{\ge0}$ we have
\[
\lim_{N\to\infty}
\mathfrak K^{(d)}_{n,\mathrm{Hn}}
\bigl(N-x,N-y\bigr)
=
\mathbb K_{\mathrm{DLaguerre}^{-}(r;\zeta+1)}(x,y).
\]
\end{prop}
\begin{proof}
Fix $j\in\ZZ$. Since
\[
\frac{2n^2}{N}\longrightarrow r>0,
\]
we have $n\to\infty$, $n/N\to0$, and
\[
\frac{(2n+j)^2}{2N} =
\frac{2n^2}{N} +
\frac{2jn}{N} +
\frac{j^2}{2N}
\longrightarrow r.
\]
Thus Theorem~6.5 of~\cite{Borodin_2017}, applied with $2n+j$
particles, lattice endpoint $2N$, and Hahn parameters
$a=b=\zeta$, gives, for every fixed $x,y\in\ZZ_{\geq0}$,
\[
\mathbb K_{2n+j,\mathrm{Hn}}(-N+x,-N+y)
\longrightarrow
\mathbb K_{\mathrm{DLaguerre}^{-}(r;\zeta+1)}(x,y).
\]
Since the Hahn weight is symmetric, the ordinary Hahn kernel is
invariant under simultaneous reflection of its arguments. Hence
\[
\mathbb K_{2n+j,\mathrm{Hn}}(N-x,N-y)
\longrightarrow
\mathbb K_{\mathrm{DLaguerre}^{-}(r;\zeta+1)}(x,y)
\]
for every fixed $j\in\ZZ$.

We now apply Proposition~\ref{As_prop_neq0} with
\[
D_N=1,\qquad
\mathbb X=\ZZ_{\leq0},\qquad
x_N(A)=N+A,\qquad
v=1,
\]
and
\[
\mathbb K_{\lim}(A,B) =
\mathbb K_{\mathrm{DLaguerre}^{-}(r;\zeta+1)}(-A,-B).
\]
The assumptions $n\to\infty$ and $N-n\to\infty$ follow from
$2n^2/N\to r>0$, and the preceding ordinary-kernel limit holds for
every fixed rank shift. Therefore Proposition~\ref{As_prop_neq0}
yields
\[
\mathfrak K^{(d)}_{n,\mathrm{Hn}}(N+A,N+B)
\longrightarrow
\mathbb K_{\lim}(A,B)
\]
for every fixed $d\in\ZZ_{\geq0}$ and $A,B\in\ZZ_{\leq0}$.
Taking $A=-x$ and $B=-y$ proves the proposition.
\end{proof}
\section*{Conclusion and future work}

We developed a rank-one comparison approach to Christoffel transformations of symmetric discrete orthogonal polynomial ensembles. One of the main results shows that the conjugated projection associated with the transformed Christoffel--Darboux kernel differs from the original orthogonal projection by a rank-one operator on the subspace of functions vanishing at the origin. This provides a general mechanism for transferring local asymptotic information from an orthogonal polynomial ensemble to its Christoffel-transformed counterpart. 
As our main application, we obtained a complete description of local fluctuations for the symplectic Schur measure~\eqref{eq:young-measure}, including the sine, Airy, discrete Hermite, and discrete hard-wall sine regimes. 
We also applied the same transfer mechanism to symmetric Hahn ensembles, obtaining sine, Airy, and hard-wall sine limits in the proportional regime and a discrete Laguerre limit in a fixed-parameter degeneration.

In a forthcoming work, we plan to study arbitrary specializations of the symplectic Schur measure~\eqref{eq:symplectic-schur-measure} using the formalism of vertex operators and free fermions. We expect to obtain similar asymptotic results, including a discrete hard-wall sine regime, as well as a symmetric Pearcey kernel analogous to that appearing in~\cite{cuenca2024symplecticschurprocess}.

The results obtained in this paper naturally lead to several questions that remain open.

\begin{itemize}
\item 
The paper studies the iterated symmetric Christoffel transformations $w(x)\mapsto x^{2d}w(x)$ supported at the origin. It would be natural to determine whether the finite-rank comparison extends to more general Christoffel transformations and, more broadly, to the full Christoffel--Geronimus--Uvarov framework~\cite{ZHEDANOV199767}.
\item
Our current argument essentially relies on the discrete nature of the underlying orthogonality measure. We do not know whether an analogous comparison technique can be developed for orthogonal polynomials on continuous intervals with absolutely continuous measures, or whether the lattice structure is an essential ingredient of the method.
\item
We also do not know how the present technique should be modified in the case of $q$-deformations. In particular, it would be interesting to understand whether analogous operator comparison results and local limit theorems remain valid for the corresponding $q$-orthogonal polynomial ensembles.
\item
It would also be interesting to reformulate our construction in terms of a symmetric discrete Riemann--Hilbert problem, in the spirit of the discrete Riemann--Hilbert formalism developed by Borodin~\cite{Borodin_2003,10.1215/S0012-7094-06-13433-6} and its later extensions. Such a reformulation could provide a more conceptual explanation of the even/odd decomposition and the Christoffel transformation.

\end{itemize}

\subsection*{Acknowledgements} 
The authors thank Dan Betea and Travis Scrimshaw for useful discussions. The work of Anton Selemenchuk was carried out with the financial support of the Ministry of Science and Higher Education of the Russian Federation in the framework of a scientific project under agreement No. 075-15-2025-013. The work of Pavel Nikitin was supported by Beijing National Science Foundation Grant No. IS26011.

\appendix
\section{Hahn particle density and edge parameters}
\label{app:Hahn_edge_parameters}

\begin{proof}[Proof of Lemma~\ref{lem:Hahn_edge_parameters}]
We translate the Hahn equilibrium-measure formulas of
\cite[\S2.4.2]{baik2007discrete} into the notation used here.
Throughout the proof, we use the abbreviations
\[
\Delta :=
\Delta_{\gamma,\kappa} =
\sqrt{
\gamma(1-\gamma)
(\kappa+\gamma)
(\kappa+\gamma+1)
},
\]
\[
q := q_{\gamma,\kappa} =
\kappa-2\gamma(\kappa+\gamma),
\qquad
r :=
r_{\gamma,\kappa} =
\kappa(\kappa+1)+2\gamma(\kappa+\gamma),
\]
\[
Z(u)
:=
Z_{\gamma,\kappa}(u)
=
(\kappa+2\gamma)\sqrt{u_{\ast}^2-u^2}, \qquad |u|\leq u_{\ast}.
\]

\medskip
\noindent
\emph{Step 1: Identification of the parameters.}
To distinguish the size parameter in \cite{baik2007discrete} from
the parameter $N$ used here, denote the former by
\[
M=2N+1.
\]
Putting $j=N+x$, we see that the nodes and the Hahn weight in
\cite[formulas~(2.69) and~(2.71)]{baik2007discrete} agree, up to an
irrelevant normalization factor, with the nodes and weight used here
under the identification
\[
P=Q=\zeta+1,
\qquad
A=B=\frac{\zeta}{M}\longrightarrow\frac{\kappa}{2},
\qquad
c=\frac{2n}{M}\longrightarrow\gamma.
\]
Below we denote the coordinate $x_{M,j}$ from
\cite[formula~(2.69)]{baik2007discrete} by $t$ in order to distinguish
it from the integer coordinate $x$ used here. It is related to the
present coordinate by
\[
t=\frac12+\frac{x}{M},
\qquad
\frac{x}{N}=2t-1+o(1).
\]
Thus, in the limiting notation of \cite{baik2007discrete},
\[
A=B=\frac{\kappa}{2},
\qquad
c=\gamma.
\]

\medskip
\noindent
\emph{Step 2: The band endpoints.}
Under the identification above, the square root of the discriminant
in \cite[formula~(2.76)]{baik2007discrete} becomes
\[
\sqrt{
\gamma(1-\gamma)
\left(\frac{\kappa}{2}+\gamma\right)^2
(\kappa+\gamma)(\kappa+\gamma+1)
}
=
\frac{\kappa+2\gamma}{2}\Delta.
\]
Consequently, \cite[formulas~(2.75) and~(2.76)]{baik2007discrete}
give
\[
\alpha
=
\frac12-\frac{\Delta}{\kappa+2\gamma},
\qquad
\beta
=
\frac12+\frac{\Delta}{\kappa+2\gamma}.
\]
In particular, $\alpha=1-\beta$. Since the limiting coordinates are
related by $u=2t-1$, the band in the $u=x/N$ coordinate is
$[-u_{\ast},u_{\ast}]$, where
\[
u_{\ast}
=
2\beta-1
=
\frac{2\Delta}{\kappa+2\gamma}.
\]

We shall also use the identities
\[
(\kappa+2\gamma)^2-4\Delta^2=q^2,
\qquad
(\kappa+1)^2(\kappa+2\gamma)^2-4\Delta^2=r^2,
\qquad
r-q=(\kappa+2\gamma)^2,
\]
which follow by direct calculation. In particular,
\[
1-u_{\ast}^2
=
\frac{q^2}{(\kappa+2\gamma)^2}.
\]

\medskip
\noindent
\emph{Step 3: The critical value and phase classification.}
In the symmetric case, the two critical values in
\cite[formula~(2.73)]{baik2007discrete} coincide:
\[
c_A=c_B
=
\frac{\sqrt{\kappa^2+2\kappa}-\kappa}{2}
=
\gamma_{\mathrm c}(\kappa).
\]
Since $q=q_{\gamma,\kappa}$, viewed as a function of $\gamma$, is
strictly decreasing and vanishes at
$\gamma=\gamma_{\mathrm c}(\kappa)$, we have
\[
q>0
\quad\Longleftrightarrow\quad
\gamma<\gamma_{\mathrm c}(\kappa),
\qquad
q<0
\quad\Longleftrightarrow\quad
\gamma>\gamma_{\mathrm c}(\kappa).
\]

Because $c_A=c_B$, the intermediate phase in
\cite[Theorem~2.17]{baik2007discrete} disappears. Thus the limiting
configuration is void--band--void for
$\gamma<\gamma_{\mathrm c}(\kappa)$ and
saturated--band--saturated for
$\gamma>\gamma_{\mathrm c}(\kappa)$. At the critical value, the last
identity in Step~2 gives
\[
u_{\ast}=1,
\qquad
\alpha=0,
\qquad
\beta=1.
\]

\medskip
\noindent
\emph{Step 4: The limiting particle density.}
Let
\[
f_{\gamma,\kappa}(t)
=
\frac{d\mu_{\min}^{\gamma}}{dt}(t)
\]
be the equilibrium density in
\cite[Theorem~2.17]{baik2007discrete}. Since the particle-to-node
ratio tends to $\gamma$, the limiting occupation density is
\[
\rho_{\gamma,\kappa}(u)
=
\gamma
f_{\gamma,\kappa}\left(\frac{1+u}{2}\right).
\]

For $|u|<u_{\ast}$, the quantities preceding
\cite[Theorem~2.17]{baik2007discrete} become
\[
T = \sqrt{\frac{u_{\ast}-u}{u_{\ast}+u}},
\qquad
k_2=k_4^{-1} =
\sqrt{\frac{1+u_{\ast}}{1-u_{\ast}}}, 
\qquad
k_1=k_3^{-1} =
\sqrt{
\frac{\kappa+1+u_{\ast}}
     {\kappa+1-u_{\ast}}
}.
\]
For $a>1$ and $T>0$,
\[
\arctan(aT)-\arctan(a^{-1}T)
=
\arctan\left(
\frac{(a-a^{-1})T}{1+T^2}
\right),
\]
where no ambiguity of branches occurs. Using
\[
\frac{T}{1+T^2}
=
\frac{\sqrt{u_{\ast}^2-u^2}}{2u_{\ast}}
\]
and the identities from Step~2, we obtain
\[
\arctan(k_2T)-\arctan(k_4T)
=
\arctan\left(
\frac{Z(u)}{|q|}
\right),
\]
\[
\arctan(k_1T)-\arctan(k_3T)
=
\arctan\left(
\frac{Z(u)}{r}
\right).
\]

Multiplying \cite[formulas~(2.77) and~(2.79)]{baik2007discrete}
by $\gamma$ and applying the two preceding identities, we find, for $|u|\leq u_{\ast}$,
\[
\rho_{\gamma,\kappa}(u)
=
\begin{cases}
\displaystyle
\frac1\pi
\left[
\arctan\left(
\frac{Z(u)}{q}
\right)
-
\arctan\left(
\frac{Z(u)}{r}
\right)
\right],
& q>0,
\\[1.4em]
\displaystyle
1-\frac1\pi
\left[
\arctan\left(
\frac{Z(u)}{-q}
\right)
+
\arctan\left(
\frac{Z(u)}{r}
\right)
\right],
& q<0.
\end{cases}
\]
Outside the band, Theorem~2.17 gives
\[
\rho_{\gamma,\kappa}(u)
=
\begin{cases}
0, & q>0,\\
1, & q<0,
\end{cases}
\qquad
u_{\ast}<|u|\leq1.
\]

Finally, at the critical value $q=0$, one has $u_{\ast}=1$.
Taking $\gamma\to\gamma_{\mathrm c}(\kappa)$ from either side in the
preceding formulas gives
\[
\rho_{\gamma,\kappa}(u)
=
\frac12
-
\frac1\pi
\arctan\left(
\frac{
(\kappa+2\gamma)\sqrt{1-u^2}
}{
r
}
\right),
\qquad
|u|<1.
\]

\medskip
\noindent
\emph{Step 5: Expansion at a right band--void edge.}
Assume $q>0$. Since $u=2t-1$ and $u_{\ast}=2\beta-1$, as
$t\uparrow\beta$ we have
\[
Z(u)
=
2(\kappa+2\gamma)\sqrt{u_{\ast}}\,
\sqrt{\beta-t}
+
O\bigl((\beta-t)^{3/2}\bigr).
\]
Using $\arctan z=z+O(z^3)$ in the density formula from Step~4 gives
\[
\rho_{\gamma,\kappa}(u)
=
\frac{
2(\kappa+2\gamma)\sqrt{u_{\ast}}
}{\pi}
\left(
\frac1q-\frac1r
\right)
\sqrt{\beta-t}
+
O\bigl((\beta-t)^{3/2}\bigr).
\]
Since $\rho_{\gamma,\kappa}=\gamma f_{\gamma,\kappa}$,
\cite[formula~(3.16)]{baik2007discrete} therefore yields
\[
\begin{aligned}
\pi\gamma B_\beta^R
&=
2(\kappa+2\gamma)\sqrt{u_{\ast}}
\left(
\frac1q-\frac1r
\right) =
\frac{
2\sqrt{2}\,
\Delta^{1/2}
(\kappa+2\gamma)^{5/2}
}{
qr
},
\end{aligned}
\]
where we used
\[
r-q=(\kappa+2\gamma)^2,
\qquad
u_{\ast}=\frac{2\Delta}{\kappa+2\gamma}.
\]

\medskip
\noindent
\emph{Step 6: The right band--saturated edge coefficient.}
Assume $q<0$. The density formula from Step~4 and
$\arctan z=z+O(z^3)$ give, as $t\uparrow\beta$,
\[
1-\rho_{\gamma,\kappa}(u)
=
\frac{
2(\kappa+2\gamma)\sqrt{u_{\ast}}
}{\pi}
\left(
\frac1{-q}+\frac1r
\right)
\sqrt{\beta-t}
+
O\bigl((\beta-t)^{3/2}\bigr).
\]
Since
\[
\frac1\gamma-f_{\gamma,\kappa}(t)
=
\frac{1-\rho_{\gamma,\kappa}(u)}{\gamma},
\]
the definition in \cite[formula~(3.18)]{baik2007discrete} yields
\[
\begin{aligned}
\pi(1-\gamma)\overline B_\beta^R =
2(\kappa+2\gamma)\sqrt{u_{\ast}}
\left(
\frac1{-q}+\frac1r
\right) =
\frac{
2\sqrt{2}\,
\Delta^{1/2}
(\kappa+2\gamma)^{5/2}
}{
(-q)r
}.
\end{aligned}
\]

Thus, in both regimes, the quantity entering the Airy coordinate is
\begin{equation}\label{eq:hahn_Airy_normalization}
\mathcal G_{\gamma,\kappa} :=
\frac{
2\sqrt{2}\,
\Delta^{1/2}
(\kappa+2\gamma)^{5/2}
}{
|q|r
} =
\begin{cases}
\pi\gamma B_\beta^R, & q>0,\\
\pi(1-\gamma)\overline B_\beta^R, & q<0.
\end{cases}
\end{equation}

\medskip
\noindent
\emph{Step 7: Conversion to the $N^{1/3}$-scale.}
By \cite[Theorems~3.7 and~3.8]{baik2007discrete}, the Airy
coordinate at the right edge is $\xi =
\bigl(M\mathcal G_{\gamma,\kappa}\bigr)^{2/3}(t-\beta)$. Since
\[
t=\frac12+\frac{x}{M},
\qquad
\beta=\frac{1+u_{\ast}}2,
\qquad
M=2N+1,
\]
we have
\[
t-\beta
=
\frac{x-\frac{M}{2}u_{\ast}}{M}
=
\frac{x-Nu_{\ast}+O(1)}{M}.
\]
For $x = Nu_{\ast} + C_{\gamma,\kappa}A N^{1/3} + O(1)$, it follows that $\xi  \longrightarrow 2^{-1/3} \mathcal G_{\gamma,\kappa}^{2/3} 
C_{\gamma,\kappa}A$. Thus $\xi\to A$ precisely when $C_{\gamma,\kappa}^{-1}
= 2^{-1/3}\mathcal G_{\gamma,\kappa}^{2/3}$.
Using~\eqref{eq:hahn_Airy_normalization}, we obtain
\[
\begin{aligned}
C_{\gamma,\kappa}^{-1} =
2^{-1/3}
\left(
\frac{
2\sqrt{2}\,
\Delta^{1/2}
(\kappa+2\gamma)^{5/2}
}{
|q|r
}
\right)^{2/3} =
\frac{
2^{2/3} \Delta^{1/3}(\kappa+2\gamma)^{5/3}
}{
\bigl(|q|r\bigr)^{2/3}
},
\end{aligned}
\]
as claimed.
\end{proof}

\section{Spectral interpretation of the first Christoffel transform} \label{sec:spectral_first_christoffel}

We explain how the spectral-projection method of Borodin and Olshanski \cite{borodin2007asymptotics,Borodin_2017} applies directly to the kernels with $d=0$ and $d=1$. These kernels correspond, respectively, to the even and odd parts of a spectral projection for the original difference operator. 

Let $\mathfrak X_N=\{-N,\ldots,N\}$ and suppose that a self-adjoint second-order difference operator $\mathcal D_N$ on $\ell^2(\mathfrak X_N)$ satisfies
\[
\mathcal D_Np_m=-\nu_{m,N}p_m,
\qquad
\nu_{0,N}<\nu_{1,N}<\cdots<\nu_{2N,N}.
\]
For a self-adjoint operator $\mathcal A$, write $[\mathcal A]_+:=\mathbf 1_{(0,+\infty)}(\mathcal A)$. For $1\leq n\leq N$ and $c_N>0$, set
\[
\mathcal A_N =
\frac{\mathcal D_N+\nu_{2n,N}\II}{c_N}.
\]
Then $[\mathcal A_N]_+$ is the orthogonal projection onto $\Span\{p_0,\ldots,p_{2n-1}\}$, and hence its matrix kernel is $\mathbb K_{2n}$.

Denote the even and odd subspaces of $\ell^2(\mathfrak X_N)$ by $\ell^2_\pm(\mathfrak X_N)$. Since $p_m(-x)=(-1)^mp_m(x)$, these subspaces reduce $\mathcal A_N$. Define the unitary maps
\[
\begin{aligned}
U_{N,+}&:
\ell^2_+(\mathfrak X_N)
\longrightarrow
\ell^2(\{0,\ldots,N\}),
&
(U_{N,+}f)(x)&=
\sqrt2\,\eta(x)f(x),\\
U_{N,-}&:
\ell^2_-(\mathfrak X_N)
\longrightarrow
\ell^2(\{1,\ldots,N\}),
&
(U_{N,-}f)(x)&=
\sqrt2\,f(x),
\end{aligned}
\]
where $\eta(x)=(1+\delta_{x,0})^{-1/2}$, and put
\[
\mathcal A_{N,\pm} =
U_{N,\pm}
\bigl(
\mathcal A_N|_{\ell^2_\pm(\mathfrak X_N)}
\bigr)
U_{N,\pm}^{-1}.
\]

\begin{prop}[Parity realization]
\label{prop:spectral_first_christoffel}
For $1\leq n\leq N$,
\[
\begin{aligned}
\mathfrak K_n^{(0)}(x,y)
&=
\left\langle
[\mathcal A_{N,+}]_+\delta_y,\delta_x
\right\rangle,
&&
x,y\in\{0,\ldots,N\},\\
\mathfrak K_n^{(1)}(x,y)
&=
\left\langle
[\mathcal A_{N,-}]_+\delta_y,\delta_x
\right\rangle,
&&
x,y\in\{1,\ldots,N\}.
\end{aligned}
\]
\end{prop}

\begin{proof}
The positive spectral subspace of $\mathcal A_N$ is spanned by $p_0,\ldots,p_{2n-1}$. Its even and odd parts are therefore $\Span\{p_0,p_2,\ldots,p_{2n-2}\}$, $\Span\{p_1,p_3,\ldots,p_{2n-1}\}$, respectively. Transporting the corresponding projections by $U_{N,+}$ and $U_{N,-}$ and using \eqref{sym_ker} and \eqref{eq:K^1_diff} gives the result.
\end{proof}

We next record the resulting convergence argument. Extend $\mathcal A_N$ to $\ell^2(\ZZ)$ by
\[
\widehat{\mathcal A}_N =
\mathcal A_N
\oplus
(-\II)_{\ell^2(\ZZ\setminus\mathfrak X_N)}
\]
and, for localization centres $X_N\in\{0,\ldots,N\}$ satisfying $N-X_N\to\infty$, set
\[
\widetilde{\mathcal A}_N =
\tau_{X_N}^{-1}
\widehat{\mathcal A}_N
\tau_{X_N},
\qquad
(\tau_Xf)(x)=f(x-X).
\]
Suppose that there is an essentially self-adjoint symmetric operator ${\mathcal A}$ on $\ell_{0}(\ZZ)$ such that
\[
\widetilde{\mathcal A}_Nf
\longrightarrow
{\mathcal A}f,
\qquad
f\in \ell_{0}(\ZZ),
\]
and that $0\notin\sigma_{\mathrm p}(\overline{{\mathcal A}})$. For tridiagonal operators, the displayed convergence amounts to coefficientwise convergence at every fixed local coordinate. By \cite[Propositions~4.2 and~4.3]{Borodin_2017},
\[
[\widetilde{\mathcal A}_N]_+
\xrightarrow{\mathrm s}
[\overline{{\mathcal A}}]_+.
\]
Consequently, if $\mathbb K_{\lim}(A,B) :=
\left\langle
[\overline{{\mathcal A}}]_+\delta_B,\delta_A
\right\rangle$, then
\[
\mathbb K_{2n}(X_N+A,X_N+B)
\longrightarrow
\mathbb K_{\lim}(A,B)
\]
for every fixed $A,B\in\ZZ$.

If $X_N\to\infty$, the reflected term vanishes. Indeed, writing $P_N=[\widetilde{\mathcal{A}}_N]_+$ and $P=[\overline{\mathcal{A}}]_+$, we have
\[
\mathbb K_{2n}(-X_N-A,X_N+B) =
\left\langle
P_N\delta_B,\delta_{-2X_N-A}
\right\rangle
\longrightarrow0,
\]
because $P_N\delta_B\to P\delta_B$ in $\ell^2(\ZZ)$ and the coordinates of the fixed vector $P\delta_B$ tend to zero at infinity. Combining this with the parity realizations above, we obtain
\[
\mathfrak K_n^{(0)}(X_N+A,X_N+B),
\quad
\mathfrak K_n^{(1)}(X_N+A,X_N+B)
\longrightarrow
\mathbb K_{\lim}(A,B).
\]

At the central point $X_N=0$, the reflection remains fixed and the image term survives:
\[
\begin{aligned}
\mathfrak K_n^{(0)}(A,B)
&\longrightarrow
\eta(A)\eta(B)
\bigl[
\mathbb K_{\lim}(A,B) +
\mathbb K_{\lim}(-A,B)
\bigr],\\
\mathfrak K_n^{(1)}(A,B)
&\longrightarrow
\mathbb K_{\lim}(A,B) -
\mathbb K_{\lim}(-A,B).
\end{aligned}
\]
These are precisely the even and odd hard-wall modifications of the limiting kernel.

The parity realization is special to the first Christoffel transform. For $d\geq2$, the relevant polynomial subspace is no longer a parity sector of the original difference operator. The higher transforms are therefore treated by the rank-one comparison in Proposition~\ref{Main_Lemma} and the transfer principle in Proposition~\ref{As_prop_neq0}. Finding a direct spectral-projection realization for these higher transforms remains an open problem.

\bibliographystyle{plainnat}
\bibliography{lit}

\end{document}